\documentclass[reqno]{amsart}
\usepackage{amsmath,amssymb}
\newtheorem{theorem}{Theorem}[section]
\newtheorem{lemma}[theorem]{Lemma}
\newtheorem{proposition}[theorem]{Proposition}
\newtheorem{corollary}[theorem]{Corollary}
\newtheorem{definition}[theorem]{Definition}
\newtheorem{remark}[theorem]{Remark}

\numberwithin{equation}{section}

\newcommand{\tq}{\tilde{\mc Q}}

\newcommand{\jn}{J^{\text{neum}}}
\newcommand{\mc}[1]{{\mathcal #1}}

\newcommand{\mb}[1]{{\mathbf #1}}
\newcommand{\bb}[1]{{\mathbb #1}}

\newcommand{\<}{\langle}
\renewcommand{\>}{\rangle}

\renewcommand{\epsilon}{\varepsilon}
 
  \newcommand{\rz}{\rho^{(0)}}

\newcommand{\ca}{\mathcal{A}}
\newcommand{\cb}{\mathcal{B}}
\newcommand{\cc}{\mathcal{C}}
\newcommand{\cd}{\mathcal{D}}

\newcommand{\cj}{\mathcal{J}}

\newcommand{\cl}{\mathcal{L}}
\newcommand{\cm}{\mathcal{M}}
\newcommand{\cn}{\mathcal{N}}

\newcommand{\cq}{\mathcal{Q}}

\newcommand{\cs}{\mathcal{S}}

\newcommand{\E}{\mathbb{E}}
\newcommand{\FF}{\mathbb{F}}

\newcommand{\N}{\mathbb{N}}

\newcommand{\R}{\mathbb{R}}

\newcommand{\V}{\mathbb{V}}

\newcommand{\Z}{\mathbb{Z}}
\newcommand{\LL}{\mathbb{L}}

\let\b=\beta
\let\d=\delta
\let\e=\epsilon
\let\ve=\varepsilon

\let\g=\gamma
\let\s=\sigma

\let\ve=\varepsilon
\let\l=\lambda

\let\r=\rho
\let\t=\theta

\let\G=\Gamma
\renewcommand{\L}{\Lambda}

\newcommand{\1}{\,\rlap{\small 1}\kern.13em 1}

\newcommand{\sqr}[2]{{\vcenter{\hrule height.#2pt%
                      \hbox{\vrule width.#2pt height#1pt\kern#1pt%
                            \vrule width.#2pt}%
                      \hrule height.#2pt}}}
\newcommand{\cqfd}{\hfill$\mathchoice\sqr46\sqr46\sqr{1.5}2\sqr{1}2$\par}

\renewcommand{\limsup}{\mathop{\overline{\hbox{\rm lim}}}}
\renewcommand{\liminf}{\mathop{\underline{\hbox{\rm lim}}}}

\newcommand{\rar}{\rightarrow}

\newcommand{\Es}{\mathbb{E}}
\newcommand{\Pb}{\mathbb{P}}

\newcommand{\bro}{{\bar{\rho}}}

\input color

\begin{document}

\title[Kawasaki process with Neumann Kac interaction] { Boundary driven Kawasaki process with long range interaction: dynamical large deviations and steady states }

\author[M. Mourragui]{Mustapha Mourragui}  \thanks{M.M. supported by ANR Blanc LHMSHE, ANR-2010-BLANC-0108, Grefi- Mefi 2010, Projet risc.} \address{Mustapha Mourragui
  \hfill\break\indent CNRS UMR 6085, Universit\'e de Rouen,
  \hfill\break\indent Avenue de l'Universit\'e, BP.12, Technop\^ole du
  Madril\-let, \hfill\break\indent
F76801 Saint-\'Etienne-du-Rouvray, France.}
\email{Mustapha.Mourragui@univ-rouen.fr}

\author[E. Orlandi]{Enza Orlandi} \thanks{E.O. supported by MURST/Cofin  
Prin2011, 
  ROMA TRE University,  Rouen University.}  \address{Enza Orlandi
   \hfill\break\indent Dipartimento di Matematica, Universit\'a di Roma Tre,  
   \hfill\break\indent L.go S.Murialdo 1, 
00146 Roma, Italy.}
\email{orlandi@mat.uniroma3.it}
\date{\today}




\begin{abstract}  
A particle system  with a single locally-conserved field (density) in a bounded interval  with different densities maintained at the two endpoints 
of the interval is under study here.
The particles interact in the bulk  through a long range potential parametrized by $\beta\ge 0$ and evolve according  to an exclusion rule.
 It is shown that the empirical   particle  density  under the    diffusive scaling
 solves a  quasi-linear  integro-differential  evolution equation 
  with Dirichlet boundary conditions.  The associated   dynamical large deviation principle  is proved.  Furthermore,  for   $\beta$ small enough,   
 it is also demonstrated that the  empirical  particle  density 
 obeys a law of large numbers with
respect to the stationary measures (hydrostatic).  The macroscopic  particle  density   solves  a  non local, stationary,  transport equation.

\end{abstract}

\maketitle
\thispagestyle{empty}

\section{Introduction}
\label{sec1}
 Over the last few  years there have  been  several papers devoted to  understanding the macroscopic properties of 
  systems out of equilibrium.  Typical examples are  systems   in contact   with two thermostats at  different temperature or   with two reservoirs  
at different  densities. 
A mathematical model is provided by  interacting particles performing  local  reversible dynamics  
(for example  reversible  hopping dynamics) in  the interior of   a domain and some external mechanism of creation and annihilation of particles at  the boundary of the domain, modeling the 
reservoirs.  The full process is non reversible.
The  stationary non equilibrium states are characterized by a flow of mass  through  the system and   long range correlations are present. 
   We refer to  \cite{BDGJL7} and  \cite {Der} for two   reviews on this topic.      
     We study a model with such   features but in a regime where   phase separation might occur   at equilibrium  for the underlying reversible model.   To this aim we consider in  the interior of   the domain a   reversible dynamics  (Kawasaki dynamics) constructed  trough  a  mean field interaction (Kac potential),  see below  for more details.  This is the first time, to our knowledge,  where  both,   long range dynamics in the bulk and   creation and annihilation of particles at  the boundary of the domain, are taken into account. 
     
   The   particle models we consider 
     are dynamic versions  of lattice  gases with long-range Kac potentials,
i.e. the interaction energy between two particles, say one at $x$ and the other at $y$ ($x$ and $y$ are both in $\Z^d$), 
is given by  $J_N (x,y)= N^{-d} J( \frac  xN, \frac  yN)$ where $J$ is a smooth  symmetric probability kernel with compact support and $N$ is a positive integer  which  is sent to $\infty$.     The equilibrium states for these models have been investigated
thoroughly  \cite {KUH},  \cite {LP}   and  \cite {PL}, and have provided great mathematical insight into the static
aspects of phase transition phenomena.
  The  dynamical  version  of these lattice gases    in domain with    periodic  boundary conditions (reversible dynamics) has been   analyzed in 
  \cite  {gl1},  \cite  {gl2}, \cite  {GLM} and \cite {AG}.  
 This paper starts by studying  the dynamics of  these systems in a bounded interval    with reservoirs (non-reversible dynamics).
We    investigate   their qualitative  behavior    in the range of the parameter when  at equilibrium phase transition is present.   
Let us   describe  informally  the dynamics. 
 We   consider  a one dimensional lattice gas   with particle reservoirs at the endpoints.
  The restriction on the dimension is done only   for simplicity.  
 Given  an integer $N>1$  at any given time each site of the discret set   $ \{-N, \dots, N \}$ is either occupied by one particle or  empty.   
The interaction energy among particles  is  given by a modified version of the Kac potential  $J_N$ and it is tuned by a positive parameter $ \beta $.
The modification of the Kac potential, see \eqref {tm1},   takes into account that the particles are confined in a     bounded domain.   
In the bulk each particle    jumps  at random times  on  the right  or  on  the left  nearest neighborhood, if the chosen site is empty,  at a rate  which 
depends  on  the  particle  configuration  through the Kac potential.   When $\beta=0$  we have the simple exclusion process.        
At the boundary sites, $\pm N$,   particles are created  and removed for the local density to be     $ 0< \rho_-  \le \rho_+ <1$.  
At rate   $\rho_\pm$ a particles is created at $\pm N$  if the site is empty and at rate $1- \rho_\pm$  the particle at $\pm N$  is removed if the site is occupied. 

 The dynamics    described  above defines an {\it irreducible}  Markov  jump process on a finite state space, i.e. there is a strictly positive  probability to go from any state to another.     By the general theory on Markov  process 
\cite{li3},   the invariant measure  $\mu_N^{stat}$  is unique  and encodes the long time behavior  of the system.  
 Let  $\Pb_{ \mu_N^{stat}}^{\b,N}$ be the stationary process.  
 Since  $ \mu_N^{stat}$ is invariant, the measure $\Pb_{ \mu_N^{stat}}^{\b,N}$   is invariant with respect to time shifts.   The measure  $\Pb_{ \mu_N^{stat}}^{\b,N}$   is invariant  under time reversal  if and only if the measure $\mu_N^{stat}$ is reversible for the process, i.e if the generator  of the process  satisfies the detailed balance condition with respect to  $\mu_N^{stat}$.
In the case at hand   $\mu_N^{stat}$ is stationary and  reversible  only if $ \beta=0$ and $ \rho^+= \rho^- = \rho$.
In  such case the invariant, reversible measure is    the Bernoulli  product measure with marginals $ \rho$. 
When $ \beta >0$  and /or $ \rho^+ \neq  \rho^-$  then
 $\mu_N^{stat}$ is  not reversible.
In such case the corresponding process is  denoted {\it non reversible}. 
We shall consider the latter case.
 The lattice space  is scaled  by $\frac 1 N$ and the time by $N^2$  ({\it diffusive} {\it limit})  while the  behavior is studied, as $N \uparrow \infty$,  of the  empirical density  of the particles    evolving according  to the dynamics described  above.  

 To prove the hydrodynamic behavior of the system, we follow the entropy
method introduced by Guo, Papanicolaou, and Varadhan  \cite {gpv}. It relies on an estimate
on the entropy    of the state of the process with respect to a reference invariant state.
The main problem in the model considered here is that the stationary state is not
explicitly known. We have therefore to consider the entropy  of the state of the process
with respect to a state which is not invariant, for instance, with respect to a product
measure with a slowly varying profile. Since this measure is not invariant, the entropy
does not decrease in time and we need to estimate the rate at which it increases.  These
estimates on the entropy   are  deduced in Subsection 3.1 

It results  that the   local empirical particle density, in the diffusive limit, 
  is the solution of   a  boundary value problem for  a   quasilinear integro differential parabolic equation,  
see  \eqref {eq:1}.
In addition to this,  it is demonstrated that  when $\beta$ is small enough,  then the empirical    particle  density 
 obeys a law of large numbers with respect to the  
stationary   measures  (hydrostatic).   The result  is obtained  characterizing  the support  
of any limit points of $\Pb_{ \mu_N^{stat}}^{\b,N}$. This intermediate   result holds for any $\beta\ge0$. 
 Then it is shown that with 
$\beta$ small enough, the stationary solution of the boundary problem  \eqref {eq:1} is unique, it   is a  global attractor for the  macroscopic evolution  with a   decay rate   uniform with respect to  the   initial datum.   This holds only for $\beta < \beta_0$  where $\beta_0>0$ depends on
 the diameter of the   domain and on the chosen interaction  $J$.
  Namely the   quasilinear non local parabolic equation  does not   satisfy  a comparison principle, which is the main tool   used in previous papers, see  for example \cite {mo2} and \cite  {flm},  to  show the hydrostatic.  
   For value larger than $\beta_0$ we  are not able to show  the uniqueness of the   stationary  solution
   of the boundary value problem \eqref {eq:1}.   
We stress that $ \beta_0< \beta_c$ where  $\beta_c$ is the value above which  phase 
segregation occurs at equilibrium;  with  the  choice  made   of  the parameters   $ \beta_c= \frac 14$, see  page 1712  of \cite{gl2}. 
 Further, we prove the dynamical large deviations for the empirical particle  density. 
 The large deviation functional is not convex as function of the density, 
it is lower semicontinuous and has compact level sets. Since the large deviation functional is not convex and the underlying dynamics does not satisfy any comparison 
principle, care needs  to be taken  to prove  the lower bound. The basic strategy to show the lower bound consists in obtaining this bound for smooth 
paths and then applying a density argument. The argument goes as follows: Given a path $ \rho$ with   finite rate functional   $ I (\rho)$  one constructs a
sequence of smooth paths 
$\rho_n$ so that $ \rho_n \to \rho$ in a suitable topology and  $ I (\rho_n) \to I (\rho)$. 
When  the large deviation  functional is   convex,  this  argument  is easily implemented. In our case, because of the lack of convexity  we modify the definition of 
the rate functional  declaring 
 it infinite if a suitable energy estimate does not hold.    In this way the  modified rate functional when finite  provides the necessary compactness to close the argument. This  method  has  been developed in    \cite {qrv} and we  adapted to our model.   The modification of the rate functional helps in showing the lower bound but makes more difficult the upper 
bound.  One needs to show that the energy estimate holds with probability super exponentially close to one.     
A similar strategy  was applied in  \cite {mo2} and \cite {blm}.
 
  In a   recent paper,   De Masi  et alii,  \cite {DPT},  
 constructed, in the phase transition regime,     stationary solutions 
  of a  boundary value  problem   equivalent to   \eqref {eq:1}  in which the density $\rho$  is replaced by the magnetization $ m= 2\rho -1$.  
They did not study the derivation of  the  boundary value problem from the particle  system   and they did not  inquire  about  the uniqueness of  stationary solutions.
They  investigated the qualitative behavior of     constructed  
stationary solutions of \eqref {eq:1} as the diameter of the domain goes to infinity.  
 They proved,  for  the constructed solution, the validity of the Fourier law    in the thermodynamic  limit showing that,  in the phase transition regime,    the
  limit equilibrium profile  has a discontinuity (which defines the position of the interface) and satisfies  a stationary free boundary Stefan problem.

The paper is organized as  follows: 
 The precise feature of the model, notations and results are stated  in  Section 2,
In Section 3, some basic estimates needed along the paper are collected.
In Section  4 the hydrodynamic and the hydrostatic limits are shown.
Section 5  is split  into 5 subsections and     deals with  dynamical large deviations.
We   prove   in Section 6   the  weak uniqueness of the  solution of  the quasi-linear  integro-differential   equation.   Furthermore,
for $\beta$ small enough,   it is shown  
that its stationary solution is unique and  it is a  global attractor in  $L^2$.

\section{Notation and Results}
\label{sec2}
 
Fix  an integer $N\ge 1$.  Call   $\L_N=\{-N ,\ldots,N\}$   and $\Gamma_N=\{-N,+N\}$  the  boundary points.
The sites of $\L_N$ are denoted by $x$, $y$ and $z$.
 The configuration  space  is $\cs_N \equiv \{0,1\}^{\L_N}$  
which we  equip   with the product topology.  Elements of  $\cs_N$, called configurations,  are denoted  $\eta$ so that  
$\eta (x)\in \{0,1\}$ stands for the number of particles at site $x$ of the configuration $\eta$. 

We  denote    $\L=(-1,1)$  ($ \bar    \L=[-1,1]$)  the macroscopic open (closed) interval,          $ \G= \{-1, 1\}$ its boundary and  $u \in  [-1,1]$  the macroscopic space coordinate.

\subsection {The interaction}  

To define the interaction between particles we introduce     
a smooth, symmetric,   translational invariant probability kernel of range 1, i.e 
    $J(u,v) = J(0, v-u)$ for all $u,v\in\R$, $J(0,u)=0$, for all $|u|>1$, and    $ \int J (u,v) dv =1$, for all $u\in \R$,   \footnote { One could take    the interaction $J$ so  that  for all $u\in \R$,   $ \int J (u,v) dv = a$  with $a>0$.
    The only difference,  with the case   at hand,  is that the  the underlying reversible  particle model  has,  at equilibrium,
    phase  transition  for $\beta \ge \b_c= \frac  1  {4  a} $, see \cite {PL}. }.
When   $ (u, v)  \in \Lambda \times \Lambda $ we define the interaction $J^{neum}(u,v)$  
imposing a reflection rule:  $u$     interacts with  $v$  and  
  with the reflected points of  $v$
 where reflections are the ones  with    respect  to  the left and right  boundary of $\L$.  For this reason
 it is referred to as ``Neuman" interaction. More precisely,   we define for  $u$ and $v$ in $\Lambda$ 
 \begin{equation}
\label{tm1}
J^{neum}(u,v):= J(u,v) + J(u, 2-v) + J(u, -2-v), 
\end{equation} 
where $2-v$  is the image of $v$ under reflections on the  right boundary $\{1\}$ and 
$-2-v$  is the image of $v$ under reflections on the  left  boundary $\{-1\}$. 
By the assumption on $J$,     $J^{neum}(u,v)=  J^{neum}(v,u)$ and  $ \int J^{neum}(u,v) dv=1$ for all $u \in \Lambda$, 
 see Lemma \ref{lem-jn}.  
We defined  the interaction \eqref {tm1}   by boundary reflections  only
for  convenience. It  has  the advantage to keep $J^{neum}$ a
symmetric  probability kernel.  This choice  of the potential  has been done already  in  \cite{DPT}.  

 The    pair interaction  between  $x$ and $y$ in $\L_N $  is     given by 
$$J_N(x,y)=   N^{-1}J^{neum} (\frac x N , \frac y N ). $$  
The total interaction energy among   particles     is given by the following   Hamiltonian 
\begin{equation}
\label{Ha1}
H_N(\eta)= -\frac{1}{2} \sum_{x,y\in\Lambda_N}  J_N(x,y) \eta (x)\eta(y)\; .
\end{equation}
   
 \subsection{ The dynamics} 

We denote by    $\eta^{x,y}$  the configuration obtained from $\eta$ by interchanging 
the values at $x$ and $y$:
\begin{equation*}
(\eta^{x,y}) (z) := 
  \begin{cases}
        \eta (y) & \textrm{ if \ } z=x\\
        \eta (x) & \textrm{ if \ } z=y\\
        \eta (z) & \textrm{ if \ } z\neq x,y,
  \end{cases}
\end{equation*}
and by 
$\sigma^x \eta$   the configuration obtained from
$\eta$ by flipping the occupation number at site $x$:
\begin{equation*}
(\sigma^{x}\eta ) (z) := 
  \begin{cases}
        1-\eta (x) & \textrm{ if \ } z=x\\
        \eta (z) & \textrm{ if \ } z\neq x.
  \end{cases}
\end{equation*}
We denote   for    $f:\cs_N\to\R$, $x,y\in\L_N$ and $\eta\in\cs_N$, 
$$ (\nabla_{x,y}f)(\eta) = f(\eta^{x,y})-f(\eta).$$
The microscopic dynamics is specified by a continuous time Markov  chain   on the state space $\cs_N$ with infinitesimal generator   given  by 
\begin{equation}
\label{gen}
\cl_N = \LL_{\b,N} + \LL_{-,N} +\LL_{+,N}  \;,
\end{equation} where for function $f:\cs_N\to\R$
\begin{equation}
\label{genO}
\left(\LL_{\b,N}  f\right) (\eta)= 
 \sum_{x \in \L_N,  y  \in \L_N\atop |x-y|=1}  C_N^\b (
x,y; \eta) \left[ (\nabla_{x,y}f)(\eta) \right] \; ,
\end{equation}
 with    rate of exchange occupancies   $C_N^\b $   given by 
\begin{equation}
\label{rate1}
C_N^\b(x,y;\eta)= 
\exp \left  \{ -\frac {\beta}2  [H_N (\eta^{x,y}) -H_N (\eta)]  \right\}\; ,
\end{equation}
where $H_N$ is the Hamiltonian \eqref{Ha1};
\begin{eqnarray*}
(\LL_{-,N} f)(\eta) &=&  \: 
c_- \big( \eta(-N) \big) \big[ f(\sigma^{-N} \eta)-f(\eta)\big], \\
  \qquad (\LL_{+,N} f)(\eta) &=& \: 
 c_+ \big( \eta(N) \big) \big[ f(\sigma^{N} \eta)-f(\eta)\big] \, ,
\end{eqnarray*}
where for any $ \rho_\pm \in (0,1)$,    $c_\pm: \{0,1\}\to \bb R$ are given by
\begin{equation*}
  c_\pm (\zeta) \;:=\;
\rho_\pm (1-\zeta) + (1-\rho_\pm) \zeta\;. 
\end{equation*}
The generator $\LL_{\b,N}$  describes the bulk dynamics which preserves the total number of particles, \ the so called   Kawasaki dynamics,  whereas 
$\LL_{\pm,N}$, which is a  generator of  a birth and death  process  acting on   $\Gamma_N$,   models the particle reservoir at the  boundary of $\L_N$. 
The   rate of the bulk dynamics   $ \{ C_N^\b(x,y;\eta), \quad x     \in \L_N,   \quad  y  \in \L_N\} $,    see \eqref {rate1},    satisfies the detailed balance with respect to the Gibbs measure associated to   \eqref {Ha1}  with chemical potential $\l \in \R$ \begin{equation}
\label{gibbs}
\mu^{\b,\l}_{N}(\eta)
= \frac 1{Z^{\b, \l}_{N}} \exp\{- \beta H_N(\eta) +\l \sum_{x\in \L_N}\eta(x) \}\;, \qquad \eta \in  \cs_N\; ,
\end{equation}
where ${Z^{\b,\l}_{N}}$ is the normalization constant.  For the bulk rates this means
$$ C_N^\b(x,y;\eta)=   e^{- \beta   [H_N (\eta^{x,y}) -H_N (\eta)] } C_N^\b(y,x;\eta^{x,y}). $$
For the generator  it means that  $\LL_{\b,N}$ is   self-adjoint w.r.t. the Gibbs measure \eqref {gibbs}, for any  $\l \in \R$.
The corresponding process is denoted reversible. 

For a positive function  $ \rho: \L_N \to    (0,1)$ we denote   by   $ \nu^{\rho}_{N} (\cdot)$  the    Bernoulli measure   with marginals \begin{equation}
\label{prod}
\nu^{\rho (\cdot)}_{N}(\eta(x)=1)=\rho(x/N)=1-\nu^{\rho}_{N}(\eta(x)=0)
 \;,\qquad   x\in\L_N, 
\eta \in  \cs_N.
\end{equation}
The   Bernoulli measure  $\nu^{\rho_+}_{N}$  ( $\nu^{\rho_-}$) is  reversible with respect to the
boundary generator  $\LL_{+,N}$   ($\LL_{-,N}$). 

The Markov process associated to  the generator    $\cl_N$, see \eqref{gen}, is  {\it irreducible} and we denote by $\mu^{stat}_N=\mu^{stat}_N(\b,\rho_-, \rho_+)$  
the unique invariant measure. 
In the notation we stress only  the dependence on  the  parameters relevant to us.
This means that  for any  $f:\cs_N\to\R$
$$ \int_{\cs}  \cl_N   f (\eta) d \mu^{stat}_N (\eta) =0,$$ 
but  the generator $\cl_N$ is {\it not} self-adjoint  with respect to  $ \mu^{stat}_N$.
The corresponding process is called {\it non reversible}.
The only case where the process is reversible is when  $ \rho_-= \rho_+$ and   $\beta =0$. In such case the product measure  associated to   $ \rho_-= \rho_+$ is invariant   and  the  process with generator  $\cl_N$ is also reversible.

We  denote by $\mc M $ the space of positive densities bounded by 1:  
\begin{equation*}
\label{dcm}
\mc M := \left\{ \rho \in L_\infty \big([-1,1],du\big) \,:\:
0\leq \rho \leq 1 \right\}\, ,
\end{equation*}
 where $du$  stands  for the integration with respect to the Lebesgue measure on $[-1,1]$. 
We equip  $\mc M $ with the topology induced by the weak convergence of
measures and  denote by $\langle \cdot,\cdot \rangle$ the  duality  mapping.  
A sequence $\{\rho^n\} \subset \mc M$ converges to
$\rho$ in $\mc M$ if and only if $$\langle \rho^n, G \rangle  = \int_\Lambda  \rho^n (u) G(u) du \to
\langle \rho, G \rangle$$  for any continuous function $G: [-1,1]\to\bb
R$.  Note that $\mc M$ is a compact Polish space that we consider
endowed with the corresponding Borel $\sigma$-algebra.
The empirical density of the configuration $\eta\in\cs_N$ is
defined as $\pi^N(\eta)$ where the map $\pi^N \colon \cs_N \to \mc
M$ is given by  
\begin{equation*}
\label{eq:2}
\pi^N (\eta) \, (u) \;:=\; 
 \sum_{x=-N+1}^{N-1} \eta (x) \,
\mb 1
\big\{ 
{\textstyle \big[ \frac{x}{N}- \frac{1}{2N}, \frac{x}{N}+ \frac{1}{2N}\big)
}
\big\} (u) \; , 
\end{equation*}
in which $\mb 1\{A\}$ stands for the indicator function of the set
$A$.  Let $\{\eta^N\}$ be a sequence of configurations with
$\eta^N\in\cs_N$. If the sequence $\{\pi^N(\eta^N)\}\subset\mc M$
converges to $\rho$ in $\mc M$ as $N\to\infty$, we say that
$\{\eta^N\}$ is \emph{associated} to  the macroscopic density profile
$\rho\in\mc M$.

\subsection {Functional Spaces} Fix a positive time $T$.
Let $D([0,T],\mc M)$,  respectively  $D([0,T],\cs_N)$,   be the set of right continuous with left
limits trajectories $\pi:[0,T]\to\mc M$, resp. $(\eta_t)_{t\in [0,T]} :[0,T]\to\mc \cs_N $,  endowed with the Skorohod
topology and   equipped with its Borel $\s-$ algebra.     Take    $\mu_N$ on $\cs_N$  and denote by   $(\eta_t)_{t\in [0,T]}$
the Markov process   with generator 
$N^{2}\cl_N $ starting,  at time $t=0$, by $ \eta_0$ distributed according to $\mu_N$.  Notice that the generator of the process has been speeded up by $N^2$.  This corresponds to the diffusive scaling.   
Denote by $\Pb_{\mu_N}^{\b,N}$ the 
probability measure on the path space $D([0,T],\cs_N)$ corresponding to 
the Markov process $(\eta_t)_{t\in [0,T]}$
and by 
$\Es_{\mu_N}^{\b,N}$ the expectation with respect to 
$\Pb_{\mu_N}^{\b,N}$.
 When $\mu_N=\delta_\eta$ for some configuration $\eta\in\cs_N$, we write simply
$\Pb_{\eta}^{\b,N}=\Pb_{\delta_\eta}^{\b,N}$ and $\Es_{\eta}^{\b,N}=\Es_{\delta_\eta}^{\b,N}$. We denote by   $\pi^N$   the map from $D([0,T],\cs_N)$  to $D([0,T],\mc M)$  defined by $ \pi^N (\eta_{\cdot})_t=  \pi^N (\eta_t)$ and by  
  $Q_{\mu_N}^{\b,N}  =  \Pb^{\beta,N}_{\mu_N} \circ (\pi^N)^{-1} $ the law of the process $\big( \pi^N (\eta_t) \big)_{t\ge 0}$. 
 For  $m\in L_\infty([-1,1])$ and $u\in\L$ we
denote   
$$
(\jn\star m)(u)=\int_\L \jn (u,v)m(v) dv.
$$
 We need some more notations.
For  integers $n$ and $m$ we denote by $C^{n,m}([0,T]\times[-1,1])$ the space of functions $G= G_t(u): [0,T] \times [-1,1] \to \R$  with $n$ derivatives in time and $m$ derivatives in space which are continuous up to the boundary.  We denote by  $C^{n,m}_0([0,T]\times[-1,1])$  the subset of $C^{n,m}([0,T]\times[-1,1])$ of   functions   vanishing  at the boundary of $[-1,1]$, i.e. $G_t(-1)= G_t(1)=0$ for $t \in [0,T]$.
We denote by   $C^{n,m}_c([0,T]\times(-1,1))$  the subset of  $C^{n,m}([0,T] \times  (-1,1))$  of functions   with compact support in $ [0,T]\times (-1,1)$.  
 \subsection {Results}
We denote   $\chi(\rho) = \rho(1-\rho)$, $\sigma(\rho)=2\chi(\rho)$ and
  $\nabla f$, resp.\ $\Delta f$,   the gradient, resp. 
the laplacian, with respect to $u$ of a function $f$. 
 For   $G\in C^{1,2}_0([0,T]\times[-1,1])$, $\rho  \in D([0,T], \mc M)$ denote
 \begin{equation}
\label{lb1}
\begin{split}
 &\ell_G^\b(\rho, \rho_0) := \big\langle \rho_T, G_T \big\rangle 
- \langle {\rho_0}, G_0 \rangle
- \int_0^{T} \!dt\, \big\langle \rho_t, \partial_t G_t \big\rangle
\\
&\qquad\qquad
- \int_0^{T} \!dt\, \big\langle  \rho_t , \Delta G_t \big\rangle
+ {\rho_+}   \int_0^{T} \! dt\, \nabla G_t(1) 
- {\rho_-}  \int_0^{T} \!dt\, \nabla G_t(-1) \\
&\qquad\qquad
-  \frac{\b}{2}\int_0^T \langle \sigma( \rho_t), (\nabla G_t)\cdot \nabla (\jn\star  \rho_t)\rangle  dt. 
\end{split}
\end{equation}
 Denote by  $\mc A _{[0, T]} \subset D\big([0,T]; \mc M\big)  $ the set 

\begin{equation*}
\mc A _{[0, T]} = \Big\{  \r \in  L^2\big([0,T],H^1(\Lambda)\big) \; :\; 
\quad \forall  G \in \mc C^{1,2}_0 ( [0,T]
 \times \L) \, ,\; \ell_G^\b(\rho, \rho_0)=0
   \Big\}.  
\end{equation*}

\begin{theorem}
\label{th-hy} 
For any sequence of initial probability measures $(\mu_N)_{N\ge 1}$, the sequence of probability measures $(Q_{\mu_N}^{\b,N})_{N\geq 1}$ 
is weakly relatively compact and  all its converging
subsequences converge to   some limit  $Q^{\beta,*}$ that is concentrated on
the absolutely continuous measures   whose density
$\rho \in   \mc A_{[0, T]}  $. Moreover, if for any $\d>0$ and for  any continuous function 
$G:[-1,1]\to \R$
$$
\lim_{N\to 0}\mu^N\Big\{ \Big| \frac1N \sum_{x\in\L_N}\eta(x)G(x/N)-\int_\L \g (u)G(u)du   \Big| \ge \delta \Big\}=0\, ,
$$
for an initial profile $\g:\L\to (0,1)$,  then
the sequence of probability measures $(Q_{\mu_N}^{\b,N})_{N\geq 1}$ 
converges to the Dirac measure  concentrated on the unique   weak  solution of the following
 boundary value problem on $(t,u)\in (0,T)\times \L $ 
\begin{equation} 
\label{eq:1}
\begin{cases}
{\displaystyle
\partial_t \rho_t (u)
+ \,  \b \nabla \cdot \Big\{  \rho_t(u) (1-   \rho_t(u)) \nabla (\jn  \star \rho_t )(u)   \Big\}
= \, \Delta \rho_t(u)}\\
{\displaystyle
\vphantom{\Big\{}
\rho_t (\mp 1) =\; \rho_\mp \quad 
\text {for } 0\le t\le T \;,
} 
\\
{\displaystyle
\rho_0(u) =\g (u) \; .
}
\end{cases}
\end{equation}
\end{theorem}
\begin {remark}   By     weak  solution of the  
 boundary value problem \eqref {eq:1} we mean $\ell_G^\b(\rho, \g)=0$ 
 for  $G \in  C^{1,2}_0([0,T]\times[-1,1])$. \end {remark}

\begin{theorem}
\label{th-hy1} There exists $\b_0$ depending on $ \Lambda$ and $\jn$ so that, for any  $\beta < \b_0$,
for any   $G \in  C^{2}_0([-1,1])$,  for any $\delta>0$,
 $$  \lim_{N \to \infty} \mu_N^{stat} \Big [ \big|  \langle  \pi_N, G\rangle- \langle \bar \rho, G\rangle\big| \ge \delta\Big] =0 \, ,$$
 where $\bar \rho$ is  the unique   weak  solution of the following
 boundary value problem 
\begin{equation} 
\label{eq:1s}
\begin{cases}
{\displaystyle
 \Delta \rho(u)-   \b \nabla \cdot \Big\{  \rho(u) (1-   \rho(u)) \nabla (\jn  \star \rho )(u)   \Big\}
= 0, \quad u \in \L,} \\
{\displaystyle
\vphantom{\Big\{}
\rho (\mp 1) =\; \rho_\mp \quad \;.
} 
 \end{cases}
\end{equation}

\end{theorem}
 We prove Theorem   \ref {th-hy} and Theorem \ref {th-hy1} in Section 4.  Recall that the stationary measure $\mu_N^{stat}$ depends on
$\b$.

\vskip0.5cm    Next we  state the large deviation principle associated to the   law of large numbers stated in Theorem \ref {th-hy}.   
Let $
\cj_{G}^\beta =  \cj_{T, G, \gamma}^\beta \colon D([0,T], \mc M)\longrightarrow
\bb R$ be the functional given by
\begin{equation}
\label{Jb1}
\cj_G^\b(\pi) := \ell_G^\b (\pi, \g)\; -\;\frac12 \int_0^{T} \!dt\,
 \big\langle \sigma( \pi_t ), \big( \nabla G_t \big)^2 \big\rangle \; ,
\end{equation}
and ${\hat I}_T^\b(\cdot | \gamma):D([0,T], \mc M)\to [0,+\infty]$   the functional defined by
\begin{equation}
\label{2:Ib}
{\hat I}_T^\b(\pi | \gamma)\; :=\; \sup_{G\in
  C^{1,2}_0([0,T]\times[-1,1])} \cj_G^\b(\pi)\; .
\end{equation}
To define the large deviation rate functional, we  introduce  the   energy functional 
$\cq :D([0,T], \mc M)\to [0,+\infty] $ given by

\begin{equation}
\label{tm3}
 \mc Q(\pi) \;    =
   \sup_{ G}  \Big\{ \int_0^Tdt \langle  \pi_t, \nabla G_t\rangle -  \frac 12  \int_0^Tdt  \langle  \sigma (\pi_t)G_t, G_t\rangle \Big\}\;, 
\end{equation}
  where the supremum is carried over all   $G \in
  C^{\infty}_c([0,T]\times (-1,1))$.   From  the concavity of  $\sigma (\cdot)$ it follows immediately  that   $\mc Q$ is convex   and therefore
lower semicontinuous.  Moreover  $\mc Q(\pi)$ is finite if and only if 
  $\pi\in L^2\big([0,T];H^1(\L)\big)$, and
\begin{equation}\label {tm4}
\mc Q(\pi) \;=\; \frac 12 \int_0^T dt\, \int_{-1}^1 du\,
\frac{(\nabla \pi_t (u))^2}{\sigma (\pi_t(u))}\;\cdot 
\end{equation}
 If \eqref {tm4} holds, then  an  integration  by parts and Schwarz inequality imply  that \eqref {tm3} is finite.
The converse  needs to be proven, for a proof of it   we refer  to  \cite{blm}, Subsection 4.1.     
The rate functional $I_T^\beta(\cdot | \gamma): D([0,T], \mc M) \to
[0,\infty]$ is given by
\begin{equation}
\label{3:Ib}
I_T^\b(\pi|\gamma) =
\begin{cases}
 {\hat I}_T^\b(\pi | \gamma) & \hbox{ if } \ \ \mc Q(\pi)<+\infty\, ,\\ 
+\infty & \hbox{ otherwise .}
\end{cases}
\end{equation}
   We show in Lemma \ref {enT1}  that $I_T^\b(\pi|\gamma) =0$  if and only if  $ \pi_t (\cdot)$  solves the  problem \eqref {eq:1}  with initial datum $\pi_0(\cdot)= \gamma(\cdot)$.

We have the following dynamical large deviation   principle.

\begin{theorem}
\label{s02}
Fix $T>0$ and an initial profile $\gamma$ in $\mc M$.  Consider a
sequence $\{\eta^N : N\ge 1\}$ of configurations associated to
$\gamma$.  Then, the sequence of probability measures $\{ Q_{\eta^N}^{\beta,N} : N\ge 1\}$ on $D([0,T], \mc M)$
satisfies a large deviation principle with speed $N$ and  rate
function $I_T^\b(\cdot|\gamma)$, defined in \eqref {3:Ib}:    \begin{eqnarray*}
&& 
\varlimsup_{N\to\infty} \frac 1N \log Q_{\eta^N}^{\beta,N}
\big( \pi^N \in \mc C\big)
\;\leq\; - \inf_{\pi \in \mc C} I_T^\b (\pi | \gamma)  
\\
&& 
\varliminf_{N\to\infty} \frac 1N \log Q_{\eta^N}^{\beta,N}
\big( \pi^N\in \mc O \big) \;\geq\; -  \inf_{\pi \in \mc O} I_T^\b (\pi |
\gamma) \;, 
\end{eqnarray*}
for any   closed set  $\mc C \subset D([0,T], \mc M)$   and    open set $\mc
O \subset D([0,T], \mc M)$.
   The functional  $I_T^\b(\cdot|\gamma)$    is  lower semi-continuos and has compact level sets.  
\end{theorem}
We prove Theorem \ref{s02} in Section 5.

\section{Basic estimate}
\label{basic}
  Next lemma  states   some properties of the potential $\jn(\cdot,\cdot)$  easily 
obtained by its definition.

\begin{lemma}\label{lem-jn}
The potential $\jn(\cdot,\cdot)$ is a symmetric probability kernel. Moreover for any regular function $G:\L\to \R$, 
we have the following:
\begin{equation}
\label{newman}
\Big|\nabla \Big( \int_\L \jn(u,v)G(v) dv\Big)\Big|\le \int_\L \jn(u,v)  \big|\nabla G(v)\big| dv.
\end{equation}
\end{lemma}

\begin {proof} 
The symmetry of $\jn$ follows immediately by the one of $J$. We have 
$$
\jn(u,v)=J(0, v-u)+J(0,2-(u+v))+J(0,2+(u+v))\, ,
$$
which is symmetric in $u$ and $v$.
We now prove that $\jn$ is probabiltity kernel. 
 Fix $u\in \L$,  
by a change of variables, 
\begin{equation*}
 \begin{aligned}
&\int_\L \jn (u,v)dv \, =\, \int_{(-1-u)  \vee {(-1)} }^{(1-u) \wedge {1}} J(0,v)d v
\, +\, \int_{(1-u) \vee {(-1)}   }^{(3-u)  \wedge {1}} J(0,v)d v \\
&\qquad\qquad\qquad\qquad
\, +\, \int_{ (-3-u)  \vee {(-1)} }^{ (-1-u) \wedge {1}} J(0,v)d v\, .
\end{aligned}
\end{equation*}
Suppose first that $u\in [0,1]$, then
\begin{equation*}
\begin{aligned}
\int_\L \jn (u,v)dv   
  &=\int_{-1}^{1-u} J(0,v)d v \, +\, \int_{1-u}^{1} J(0,v)d v\\
\ & = \int_{-1}^{1} J(0,v)d v \, =\, 1.
\end{aligned}
\end{equation*}
and thus, $\jn(u,\cdot)$ is a probability. The proof for $u\in [-1,0]$ is similar.
It remains to prove \eqref{newman}:
\begin{equation*}
 \begin{aligned} &
\nabla \Big( \int_\L \jn(u,v)G(v) dv\Big)  = \int_\L \partial_u \jn(u,v)G(v)dv\\
\ &=\int_\L \Big[ J(u,v)- J(u,2-v) -J(u,-2-v)  \Big]\nabla G(v) dv \, . 
\end{aligned}
\end{equation*}
The result follows from the following inequatity,
$$
\Big|J(u,v)- [J(u,2-v) +J(u,-2-v)]\Big| \le 
\jn(u,v)
$$
for all $u,v\in \L$.
\end {proof}

\medskip 

\noindent  For any $G:\L\to \R$ and $x,x+1\in\L_N$ denote by   $\nabla^N G (\frac x N)$
  the discrete gradient: 
\begin{equation}
\label{derive1}
\nabla^N G (x/N)=N\big[G ((x+1)/N)-G(x/N)\big]\;.
\end{equation}

Next, we show  that  the  rate $C_N^\b$  of $ \LL_{\b,N}$     is a perturbation of the  rate of the   symmetric simple exclusion generator.   
\medskip
\begin{lemma}
\label{b1}
For any $x\in\L_N$, with $x \pm 1\in\L_N$ and $\eta\in \cs_N$,
$$
C_N^\b(x,x \pm 1;\eta)= 1 \mp \frac\b2 \big(\eta(x + 1)-\eta(x) \big)  N^{-1}\nabla^N\big[ \big( \jn\big) \star \pi(\eta)\big](x/N)+O(N^{-2})\; .
$$
 \end{lemma}

\begin{proof} 
By definition of $H_N$, for all $x,y\in\L_N$ and $\eta\in\cs_N$,
\begin{equation*}
\begin{aligned}
&\big(\nabla_{x,y}H_N\big)(\eta) =
\frac{1}{N}\big(\eta(x)-\eta(y) \big)^2\big( \jn(\frac xN, \frac yN ) - \jn(0,0)  \big) \\
&\ \qquad +
\big(\eta(x)-\eta(y) \big) \frac{1}{N}\sum_{z\in\L_N}\eta(z) \big [ \jn(\frac xN, \frac zN ) -\jn(\frac yN, \frac zN ) \big].\\ 
\end{aligned}
\end{equation*}
This concludes the proof. 
\end{proof}

\vskip0.5cm \noindent
We start recalling the definitions  of  relative  entropy and Dirichlet form, that are the main tools in the \cite{gpv} approach. 
Let $h:\L\to (0,1)$ and $\nu_N^{h(\cdot)}$ be  the   product   Bernoulli measure  defined in \eqref{prod}. Given    $\mu$,  a probability measure on $\cs_N$,  denote by
$H(\mu|\nu_N^{h(\cdot)} )$ the relative entropy of $\mu$ with
respect to $\nu_N^{h(\cdot)}$:
$$
H(\mu|\nu_N^{h(\cdot)}) \; =\; \sup_{f} \Big\{
\int f(\eta) \mu (d\eta) - \log \int e^{f(\eta)} \nu_N^{h(\cdot)} (d\eta) \Big\}\; , 
$$
where the supremum is carried over all bounded
functions on $\cs_N$. Since $\nu_N^{h(\cdot)}$ gives a
positive probability to each configuration, $\mu$
is absolutely continuous with respect to $\nu_N^{h(\cdot)}$ and  we
have an explicit formula for the entropy: 
\begin{equation}
\label{ent1}
H(\mu|\nu_N^{h(\cdot)} ) \; =\; \int  
\log  \Big\{ \frac{d\mu }{d \nu_N^{h(\cdot)}} \Big\} \, d\mu \; .
\end{equation}
Further, since  there is at most one particle per site, there exists a constant $C$, that depends only on  $h(\cdot)$,  such that
\begin{equation}
\label{entbound}
H(\mu| \nu_N^{h(\cdot)}  ) \;\leq\; C  N
\end{equation}
for all probability measures $\mu$ on $\cs_N$  (cf. comments
following Remark V.5.6 in \cite{kl}).

\subsection{Dirichlet form estimates}
One of the main step  for   deriving  the  hydrodynamic limit and   the large deviations,  is a super exponential estimate 
which allows the replacement of local functions by functionals of the empirical density. One needs to estimate expression such as 
$\langle Z,f \rangle_{\mu^N}$ in terms of Dirichlet form $ N^2\langle -\cl_N   \sqrt{f(\eta)},  \sqrt{f(\eta)} \rangle_{\mu^N}$, where $Z$ is a local function and $\langle\cdot,\cdot\rangle_{\mu^N}$ 
represents  the inner  product with respect to some probability  measure  $ \mu^N$.  In the context of boundary driven process, the fact that the  invariant measure   is not 
explicitly known introduces a technical difficulty. 
We fix as reference measure a  product measure $\nu_N^{\theta(\cdot)} $, see \eqref {prod},  where $\theta$ is  a smooth function with the   only requirement that $\theta(\mp 1) \; =\; \rho_{\mp}$. 
There is therefore  no reasons for  
$ N^2 \langle -\cl_N   \sqrt{f(\eta)},  \sqrt{f(\eta)}\rangle_{\nu^\theta (\cdot)}$ to be positive.
Next  lemma  estimates this quantity.

Define the following  functionals
from    $h \in L^2 ( \nu)$ to $\R^+$: 
\begin{equation}
\label{dir-form2}
\begin{aligned}
\cd_{0,N}\big(h, \nu\big) & =
 \sum_{x=-N}^{N-1} \int
  \left( {h}(\eta^{x,x+1}) -{h}(\eta) \right)^2 d \nu (\eta) \, ,\\
\cd_{+,N} \big(h,\nu\big) & =  \frac12
 \int  c_+ \big( \eta(N) \big) 
  \left( {h}(\sigma^{N-1}\eta) -{h}(\eta) \right)^2 d \nu (\eta)\, , \\
\cd_{-,N} \big(h,\nu\big) & = \frac12 
  \int  c_- \big( \eta(-N)  
  \left( {h}(\sigma^{-N+1}\eta) -{h}(\eta) \right)^2 d \nu (\eta) \, .
\end{aligned}
\end{equation}
 
\begin{lemma}
\label{dirichlet} 
Let $\theta:\overline\L\to (0,1)$ be a smooth function such that
$\theta(\mp 1) \; =\; \rho_{\mp}$.
There exists a positive 
constant $C_0\equiv C_0(\|\nabla \theta \|_\infty )$  so  that for any   $a >0$  and for $f\in L^2\big( \nu_N^{\theta(\cdot)}  \big)$,
 
\begin{equation}\label {dirt1}
 \int_{\cs_N} f(\eta) \LL_{\b,N}
f(\eta)  d  \nu_N^{\theta(\cdot)} (\eta) 
\le -\big(1-{a}\big) {\cd}_{0,N} \big({f},\nu_N^{\theta(\cdot)} \big)  
+\frac{C_0}{a} N^{-1}  \| f\|^2_{L^2(\nu_N^{\theta(\cdot)}) },
\end{equation}
\begin{equation}\label {dirt2} \int_{\cs_N} f(\eta) \LL_{\pm,N}
f(\eta)  d \nu_N^{\theta(\cdot)}(\eta)  
= -  {\cd}_{\pm,N} \big({f},\nu_N^{\theta(\cdot)} \big) \; . 
\end{equation}
\end{lemma}

\begin{proof}
The proof of \eqref {dirt2} follows from the reversibility of the Bernoulli measure  $\nu_N^{\theta(\cdot)}$
 with respect to $\LL_{\pm,N}$.
Next, we show \eqref {dirt1}. By Lemma \ref{b1},
\begin{equation}
\label{l0}
\begin{aligned}
\int_{\cs_N} f(\eta) \LL_{\b,N}
f(\eta)  d  \nu_N^{\theta(\cdot)} (\eta) &\le 
 \sum_{x=-N}^{N-1} \int\Big[\big(\nabla^{x,x+1} f\big)(\eta)\Big]f(\eta) 
d\nu_N^{\theta(\cdot)} (\eta)\\
\ &\ 
+\frac{A_1}{N}\sum_{x=-N}^{N-1} \int\Big|\big(\nabla^{x,x+1} f\big)(\eta)\Big|f(\eta) 
d\nu_N^{\theta(\cdot)} (\eta)
\end{aligned}
\end{equation}
for some positive constant $A_1$ depending only on $\beta$ and $J$.
We write the first term of the right hand side of  \eqref {l0}   as
\begin{equation}
\label{l1}
\begin{aligned}
&-\sum_{x=-N}^{N-1} \int\Big[\big(\nabla^{x,x+1} f\big)(\eta)\Big]^2
d\nu_N^{\theta(\cdot)} (\eta)\\
&\
+ \sum_{x=-N}^{N-1}\int R_N(x,x+1;\theta,\eta)
\Big[\big(\nabla^{x,x+1} f\big)(\eta)\Big]f(\eta^{x,x+1})
d\nu_N^{\theta(\cdot)} (\eta)\; ,
\end{aligned}
\end{equation}
where
\begin{equation}
\label{reste}
R_N(x,x+1;\theta,\eta)=
\big[ 1- e^{-N^{-1}\nabla^N\l(\theta(x/N))
 \big(\nabla^{x,x+1}\eta(x)\big)} \big]\; ,
\end{equation}
$\lambda$ is the chemical potential defined by
\begin{equation}
\label{chemical}
\lambda(r)=\log\left[ r/(1-r)\right]\; 
\end{equation}
and $\nabla^N$ stands for the discrete derivative defined in \eqref{derive1}. 
 By the inequality 
\begin{equation}\label{sch-ineq}
\text{for all}\ A,B\in\R\quad  and \quad a>0\; ,\quad
AB\le \frac{a}{2}A^2 +\frac{1}{2a}B^2
\end{equation}
and Taylor expansion, the formula \eqref{l1} is bounded by
\begin{equation}
\label{l2}
- \big(1-\frac{a}2\big) \sum_{x=-N}^{N-1} \int\Big[\big(\nabla^{x,x+1} f\big)(\eta)\Big]^2
d\nu_N^{\theta(\cdot)} (\eta)
+\frac{A_2}{a }N^{-1}\| f\|^2_{L^2(\nu_N^{\theta(\cdot)}) }
\end{equation}
for all $a>0$. Here $A_2$ is a positive constant.

The second term  on   the right hand side of  \eqref {l0}  is handled in the identical way. It  is bounded by
\begin{equation}
\label{l3}
 \frac{a}2 \sum_{x=-N}^{N-1} \int\Big[\big(\nabla^{x,x+1} f\big)(\eta)\Big]^2
d\nu_N^{\theta(\cdot)} (\eta)
+ \frac{A_3}{a} N^{-1}\| f\|^2_{L^2(\nu_N^{\theta(\cdot)}) }\; .
\end{equation}
The lemma follows from \eqref{l0},\eqref{l1}, \eqref{l2} and \eqref{l3}.
\end{proof}

Denote   for $h \in L^2 ( \nu)$
$$
\cd_{\beta,N}\big(h, \nu\big) =
 \sum_{x=-N}^{N-1} \int C_N^\b(x,x+1;\eta)
  \left( {h}(\eta^{x,x+1}) -{h}(\eta) \right)^2 d \nu (\eta) \, .
$$
\begin{lemma}
\label{dir0}
There exists a positive constant $C_1= C_1 (\beta,J)$ such that, for  any measure $\nu$ and for $h \in L^2 ( \nu)$,
$$
\big(1-\frac{C_1}N\big)\cd_{0,N}\big(h, \nu\big)\le \cd_{\beta,N}\big(h, \nu\big) 
\le \big(1+\frac{C_1}N\big) \cd_{0,N}\big(h, \nu\big).
$$
\end{lemma}

\begin{proof}
The proof is elementary since $ \big|C^\b_N(x,x+1,\eta)-1\big|$ is uniformly bounded in $N$, $x$ and $\eta$.
\end{proof}

\begin{lemma}
\label{dir2}
Let $\rho,\rho_0:\overline\L\to (0,1)$ be two smooth functions.
There exists a positive 
constant $C_0'\equiv C_0'(\|\nabla \rho_0\|_\infty ,\|\nabla \rho\|_\infty )$ such that for any probability  measure $\mu^N$ on $\cs_N$,
\begin{equation}
\label{bound-l1}
 \cd_{0,N} \Big(\sqrt{    \frac{d\mu^N}{d \nu_N^{\rho(\cdot)} }    },
\nu_N^{\rho(\cdot)} \Big)\; \le \;
2\; \cd_{0,N} \Big(\sqrt{ \frac{d\mu^N}{d\nu_N^{\rho_0(\cdot)} }},
\nu_N^{\rho_0(\cdot)}  \Big)
  \; +\; C_0' N^{-1}\; .
\end{equation}
\end{lemma}

\begin{proof}
Denote by $f(\eta)=\frac{d\mu^N}{d\nu_N^{\rho(\cdot)}}(\eta) $
and $h(\eta)=\frac{d\mu^N}{d\nu_N^{\rho_0(\cdot)}}(\eta)$. Since  
$f(\eta)=  h(\eta)  \frac { {d\nu_N^{\rho_0(\cdot)}}(\eta) }  {d\nu_N^{\rho(\cdot)}(\eta)} $ we obtain for 
$-N \le x\le N -1$ the following  
\begin{equation*}
\begin{aligned}
&\int_{\cs_N}  \Big[ \nabla_{x,x+1}\sqrt{f}(\eta) \Big]^2
d\nu_N^{\rho (\cdot)}(\eta) \\
\ & \ \ 
=\int_{\cs_N}  \Big[ \sqrt{h}(\eta^{x,x+1})
R_2(x,x+1;\eta) +\nabla_{x,x+1} \sqrt{h}(\eta) \Big]^2
d\nu_N^{\rho_0(\cdot)}(\eta)  \\
&\ \ \ 
 \le 2 \int_{\cs_N}  \Big[ \nabla_{x,x+1}\sqrt{h}(\eta) \Big]^2
d\nu_N^{\rho_0 (\cdot)}(\eta) \\
\ &\quad
+2\int_{\cs_N} h(\eta^{x,x+1})
\big[R_N(x,x+1;\rho,\eta)\big]^2 
d\nu_N^{\rho_0(\cdot)}(\eta) \; ,
\end{aligned}
\end{equation*}
where
$$
R_2(x,x+1;\eta) = 
\exp \big\{ (1/2)N^{-1} \nabla^N [ \l(\rho(x/N))-\l(\rho_0(x/N))  ]  \nabla_{x,x+1}\eta(x) \big\} -1 \; 
$$
and $\lambda$ is the chemical potential defined by \eqref{chemical}.
We conclude the proof using Taylor expansion and integration by
parts.
\end{proof}

\subsection{Superexponential estimates}

For a positive integer $\ell$ and $x\in\L_N$   denote  
\begin{equation*}
\Lambda_{\ell}(x) = \Lambda_{N,{\ell}}(x) = \{y\in\L_N:\, |y-x|\leq {\ell}\}\,.
\end{equation*}
When $x=0$, we shall denote $\L_\ell(0)$ simply by $\L_\ell$, that is, for all $1\le \ell \le N$,
$$
\Lambda_\ell \equiv \Lambda_{N,\ell}(0) = \{-\ell,\cdots,\ell\}\, . 
$$
Denote the empirical
mean density on the  box  $\Lambda_{\ell}(x)$    by $\eta^{\ell}(x)$:
\begin{equation}
\label{average}
\eta^{\ell}(x) = \frac{1}{|\Lambda_{\ell}(x)|}\sum_{y\in\Lambda_{\ell}(x)}\eta(y)\,.\end{equation}
For a cylinder function $\Psi$,   that is a function on $\{0,1\}^\Z$ depending on $\eta (x)$, $ x \in \Z$, only trough finitely many $x$,   denote  by $\widetilde{\Psi}(\rho)$ the expectation of $\Psi$ with respect to  $\nu^{\rho}$, the Bernoulli 
product  measure   with density $\rho$:
\begin{equation}  \label {te1} 
\widetilde{\Psi}(\rho) = E^{\nu^{\rho}}[\Psi]\, .
\end{equation}
  Further, denote  for   $G\in\mc C([0,T]\times [-1,1])$ 
and $\varepsilon>0$ 
\begin{equation}\label{vne}
V_{N,\varepsilon}^{G,\Psi}(s,\eta)=\frac{1}{N}\sum_{x \in \Lambda_N}
G_s(x/N)\left[\tau_x\Psi(\eta)-
\widetilde{\Psi}(\eta^{[\varepsilon N]}(x))\right]\, ,
\end{equation}
where     the sum is carried over all $x$ such that the support of
$\tau_x\Psi$ belongs to $\L_N$ and   $ [\cdot] $ denotes the  lower integer part.

\begin{proposition}
\label{see1}
Let  
$\{\mu_N: N\geq 1\}$ be  a sequence of probability measures on $\cs_N$. For
every $\delta>0$,
\begin{equation*}
\limsup_{\varepsilon\to 0}\limsup_{N\to\infty}
\frac{1}{N}\, \log \Pb_{\mu_N}^{\b,N}
\Big[ \, \Big|\int_0^T V_{N,\varepsilon}^{G,\Psi}(s,\eta_s) \, ds \Big|
>\delta\Big] \;=\; -\infty\, .
\end{equation*}
\end{proposition}

\begin{proof}
 Fix    $c > 0$ that will decreases to $0$ after $\e$   and a smooth function $\rho_c\colon \L\to (0,1)$ 
 which is constant  in $\L_{(1-c)}\, =\, [-1+c, 1-c]$ and   equal to $\rho_\pm$ at the boundary, i.e   $\rho_c(\pm 1)=\rho_\pm$.    The constant can be arbitrarily chosen and we  denote it $\g_0$. 
Divide  $\L_{N}$ in two subsets, $\L_{[(1-2c)N]}$
    and $\L_{N}\setminus \L_{[(1-2c)N]}$ and split  the sum over $x $   in the definition of  $V_{N,\varepsilon}^{G,\Psi}$ into the sum  over these two sets.   Since 
$$
\sup_{\eta, \varepsilon,N,x\in\L_N} \Big\{G_s(x/N)\left[\tau_x\Psi(\eta)-
\widetilde{\Psi}(\eta^{[\varepsilon N]}(x))\right]\Big\}<\infty\ ,
$$ 
we have that
\begin{equation}\label{at1}
 \left | \frac{1}{N}\sum_{x \in \L_{N}\setminus \L_{[(1-2c)N] }}
\int_0^T ds G_s(x/N)\left[\tau_x\Psi(\eta)-
\widetilde{\Psi}(\eta^{[\varepsilon N]}(x))\right] \right |  \le  cTK_0 
\end{equation}
for some positive constant $K_0$ which depends on $G$.
By Chebyshev exponential  inequality,  for all $a>0$,
\begin{equation}
\label{es1}
\begin{aligned} 
&\frac{1}{N}\, \log  \Pb_{\mu_N}^{\b,N}
\Big[ \, \Big|\int_0^T V_{N,\varepsilon}^{G,\Psi}(s,\eta_s) \, ds \Big|
>\delta\Big]\\
&\qquad\qquad
\le -a\big(\delta  - T K_0 c\big)\; +\; 
\frac{1}{N}\, \log  \Pb_{\mu_N}^{\b,N}
\Big[ \,  \exp  \left (  aN\, \Big|\int_0^T \V_{N,\varepsilon}^{c,G,\Psi} (s,\eta_s) \, ds \Big| \right )  \Big]\; ,
\end{aligned}
\end{equation}
  where
\begin{equation}
\label{at2}
\V_{N,\varepsilon}^{c,G,\Psi} (s,\eta)
\; =\;
\frac{1}{N}\sum_{x\in \L_{[(1-2c)N]}}
G_s(x/N)\left[\tau_x\Psi(\eta)-
\widetilde{\Psi}(\eta^{\varepsilon N}(x))\right]\; .
\end{equation}
  It is immediate to see that  the  Radon-Nikodym derivative  
$$\frac{d  \Pb^{\b,N}_{\mu_N} }{d\Pb^{\b,N}_{\nu_N^{\rho_c (\cdot )} }}\big( (\eta_t)_{t\in [0,T]} \big)=
\frac{d \mu_N}{d\nu_N^{\rho_c (\cdot )}}  \le  e^{ N K_1(c)} $$
   for some positive $K_1(c)$
  that depends on $c$.  
   The right hand side of \eqref{es1} is bounded by
\begin{equation}
\label{te2}
-a\big(\delta  - T K_0 c\big)\; +\; K_1(c)\; +\; 
\frac{1}{N}\, \log \Pb^{\b,N}_{\nu_N^{\rho_c (\cdot )} } 
\Big[ \, \ \exp  \left (  aN\, \Big|\int_0^T \V_{N,\varepsilon}^{c,G,\Psi} (s,\eta_s) \, ds \Big| \right )  \Big]\; .
\end{equation}
  Since $e^{|x|}\le e^{x}+e^{-x}$ and 
\begin{equation}\label{maxlim}
\limsup N^{-1} \log \{a_N + b_N\}
\le  \max \{ \limsup  N^{-1} \log a_N\; , \;\limsup  N^{-1} \log b_N \},
\end{equation}
we may remove the absolute value in the third term of 
\eqref{te2}, provided our estimates
remain in force if we replace $G$ by $-G$. 
Denote by 
$$ (\cl_N)^s = \frac 1 2 \left ( \cl_N + \cl_N^{\star} \right ) $$
where  $ \cl_N^{\star}$ is the adjoint of  $ \cl_N$ in  $L^2(\nu_N^{\rho_c (\cdot )})$. 
By the Feynman-Kac 
formula, 
 \begin{equation}
\label{te3} \frac{1}{N}\, \log \Pb^{\b,N}_{\nu_N^{\rho_c (\cdot )} } 
\Big[ \,  exp  \left (  aN\,  \int_0^T \V_{N,\varepsilon}^{c,G,\Psi} (s,\eta_s) \, ds  \right )  \Big]\;  \le \frac{1}{ N}\int_0^T \lambda_{N,\epsilon}(G_s) \, ds\; ,
\end{equation}
where $\lambda_{N,\epsilon}(G_s)$ is the largest eigenvalue of  $\{N^2 (\cl_N)^s + Na
\V_{N,\varepsilon}^{c,G,\Psi} 
( G_s,\cdot)\}$. By the variational formula for the
largest eigenvalue, 
for each $s\in [0,T]$,  
 \begin {equation}
\label{tes1}  \frac 1 N 
\lambda_{N,\epsilon}(G_s)= 
\sup_{f} \Big\{\int a\V_{N,\varepsilon}^{c,G,\Psi} \big( G_s,\eta\big) f(\eta)
\nu_N^{\rho_c(\cdot)}(d\eta) \; +\; N \langle  \cl_N \sqrt{f} , 
\sqrt{f} \rangle_{\nu_N^{\rho_c (\cdot )}} \Big \}\; .
\end{equation}
In this formula the supremum is carried over all densities $f$
with respect to $\nu_N^{\rho_c (\cdot )}$.
 By  Lemma \ref{dirichlet}, since  \eqref {dirt2} gives a negative contribution,   it is enough,   to  get the result,  to choose $c$ such that $c<\frac{\delta}{TK_0}$ and
to show that, there exists $M>0$ that depends only on $G$ and $c$, such that, for all $a>0$
$$
\limsup_{\varepsilon \to 0}\limsup_{N\to \infty}
\sup_{f} \Big\{\int a\V_{N,\varepsilon}^{c,G,\Psi} \big( G_s,\eta\big) f(\eta)
\nu^{\rho_c (\cdot )}_N(d\eta) \; -\;  N\cd_{0,N}(\sqrt{f}, \nu_N^{\rho_c (\cdot )})\Big\}
\; \le\; M \;.
$$
We then let $a\uparrow \infty$.
Notice that for $N$ large enough the function $ \V_{N,\varepsilon}^{c,G,\Psi} \big( G_s,\eta\big) $ depends on the configuration $\eta$ only through
the variables $\{ \eta(x),\; x\in \L_{(1-c)N}\}$.    Since
$\rho_c$ is  equal to $\gamma_0$, in
$\L_c$, we   replace $\nu_N^{\rho_c (\cdot)}$
in the previous formula by $\nu^{\gamma_0}_N$ with respect to   which  the operator $\LL_{0,N}$
is reversible.    Therefore $ \cd_{0,N}(\cdot\, ,
 \nu^{\gamma_0}_N)$ is the Dirichlet form associated to the generator
 $\LL_{0,N}$.  Since the Dirichlet form is convex, it remains to show that
$$
\limsup_{\varepsilon \to 0}\limsup_{N\to \infty}
\sup_{f} \Big\{\int a\V_{N,\varepsilon}^{c,G,\Psi} \big( G_s,\eta\big) f(\eta)
\nu^{\gamma_0}_N(d\eta) \; -\;  N\cd_{0,N}(\sqrt{f}, \nu_N^{\gamma_0 (\cdot )})\Big\}
\; =\; 0,
$$
for any $a>0$.
This follows from the usual one block and two blocks  estimates
(cf. Chap 5 of  \cite{kl}).
\end{proof}

For $x=\pm N$, a configuration $\eta$ and $\ell \ge 1$, let
\begin{equation}\label{vneb}
W_{N}^{\pm,\ell}(\eta) = \left  | \eta^\ell(\pm N) - \rho_\pm\right |,
\end{equation}
  where,   see  \eqref {average},  $\eta^\ell(N) =\frac 1 {  \ell +1}\big\{\eta(N-\ell)+\cdots+\eta(N) \big\}  $
     and $\eta^\ell(-N) =\frac 1 {  \ell +1}\big\{\eta(-N+\ell)+\cdots+\eta(-N) \big\}  $.

\begin{proposition}
\label{seeb}
Fix a sequence
$\{\mu_N: N\geq 1\}$ of probability measures on $\cs_N$. For every $\delta>0$,
\begin{equation*}
\limsup_{\ell\to\infty}\limsup_{N\to\infty}\frac{1}{N} \, \log  \Pb^{\b,N}_{\mu_N}  \Big[
\,\int_0^T  W_{N}^{\pm,\ell}(\eta_s) ds  \, >\delta 
\Big] \;=\; -\infty\, .
\end{equation*}
\end{proposition}

\begin{proof}
Consider first the limit  with the term  $W_{N}^{+,\ell}$.
Fix a smooth function $\gamma:\overline{\L}\to (0,1)$ such that $\gamma(-1)=\rho_-$, and $\gamma(u)=\rho_+$ for
$u\in [0,1]$. Since the Radon-Nikodym derivative 
$\frac{d \mu_N}{d\nu_N^{\g (\cdot )}}$ is bounded by $\exp(N K_1)$ for some positive
constant $K_1$, it is enough to show that
$$
\limsup_{\ell\to\infty}
\limsup_{N\to\infty}\frac{1}{N} \, \log  \Pb^{\b,N}_{\nu_N^{\gamma(\cdot)}}  \Big[
\,\int_0^T W_{N}^{+,\ell}(\eta_s)  ds 
\, >\delta 
\Big] \;=\; -\infty\, .
$$
We   follow the same steps as in Proposition \ref {see1}.  Applying   Chebyshev exponential inequality and Feynman-Kac formula, 
the expression in the last limit is bounded for all $a>0$ by
\begin{equation} \label {et4a}
-a\delta\; +\;  
\frac{T}{ N} \, \widetilde \lambda_{N,\epsilon}(a) \; ,
\end{equation}
where for all $a>0$, $\frac{1}{N} \widetilde\lambda_{N,\epsilon}(a)$ is the largest eigenvalue of the 
$\nu_N^{\gamma (\cdot )}$-reversible operator 
$$
f\to N (\cl_N)^s (f)+ a  
\Big(W_{N}^{+,\ell} \Big) f .
$$ 
Here $(\cl_{N})^s$ is the symmetric part of the operator $\cl_N$
in $L^2(\nu_N^{\gamma (\cdot )})$. By the variational formula for the
largest eigenvalue, we have
\begin{equation*}
\begin{aligned}
&\frac{1}{N}\widetilde\lambda_{N,\epsilon}(a)\,=\,
\sup_{f} \Big\{\int a W_{N}^{+,\ell}(\eta)  f(\eta)
\nu_N^{\gamma(\cdot)}(d\eta) \\
&\qquad\qquad\qquad\qquad\qquad\qquad\qquad
\; +\;  N \, < \cl_N\sqrt{f} , 
\sqrt{f} >_{\nu_N^{\gamma (\cdot )}} \Big \}\; .
\end{aligned}
\end{equation*}
In this formula the supremum is carried over all densities $f$
with respect to $\nu_N^{\gamma (\cdot )}$.
By Lemma \ref{dirichlet}, we just need  
to show that, there exists $M>0$, such that, for all $a>0$
\begin{equation*}
\begin{aligned}
&\limsup_{\ell \to \infty}\limsup_{N\to \infty}
\sup_{f} \Big\{\int a W_{N}^{+,\ell}(\eta)  f(\eta)
\nu_N^{\gamma(\cdot)}(d\eta)\\
&\qquad\qquad\qquad\qquad\qquad
\; -\;  N\cd_{0,N}(\sqrt{f}, \nu_N^{\gamma (\cdot )})
-N \cd_{+,N}(\sqrt{f}, \nu_N^{\gamma (\cdot )})\Big\}
\; \le\; M \; .
\end{aligned}
\end{equation*}
Recall that the profile $\gamma$ is constant equal to $\rho_+$ on $[0,1]$.  Since  $W_{N}^{+,\ell}(\eta)$ depends only on coordinates in a box
$\Lambda_\ell(N)$, we   replace $\nu_N^{\gamma (\cdot)}$
in the previous formula by $\nu^{\rho_+}_N$. On the other hand, $\nu^{\rho_+}_N$ is
reversible for $\LL_{0,N}+ \LL_{+,N}$   and therefore $ \cd_{0,N}(\cdot\, ,
\nu^{\rho_+}_N)+ \cd_{+,N}(\cdot\, ,\nu^{\rho_+}_N)$ is the Dirichlet form associated to the generator
$\LL_{0,N}+\LL_{+,N}$. Since the Dirichlet form is convex, it remains to show that
\begin{equation*}
\begin{aligned}
&\limsup_{\ell \to \infty}\limsup_{N\to \infty}
\sup_{f} \Big\{\int a W_{N}^{+,\ell}(\eta)  f(\eta)
\nu_N^{\rho_+}(d\eta)\\
&\qquad\qquad\qquad\qquad\qquad
\; -\;  N\cd_{0,N}(\sqrt{f}, \nu_N^{\rho_+})
-N \cd_{+,N}(\sqrt{f}, \nu_N^{\rho_+})\Big\} 
\; =\; 0.
\end{aligned}
\end{equation*}
for any $a>0$.
This follows from the law of large numbers by applying  the same device  used in the proof of the one block and two blocks estimates,
(cf. Chap 5 of \ \cite{kl},  and Lemma 3.12, Lemma 3.13 in \cite{mo2}).
\end{proof}
\subsection{Energy estimate} We prove in this subsection an energy estimate  which is   one of the main ingredient in the proof of large 
deviations and hydrodynamic limit. It allows to prove Lemma \ref{lem3} and 
to exclude paths with infinite energy in the large deviation regime. 
For $\delta >0$, $G\in C^\infty_c ([0,T]\times \L)$ 
define 
 \begin{equation}\label{cccg}
\tq^\d_G(\pi)\;=\;
\int_0^T dt\<\pi_t, \nabla G_t \> -\d \int_0^T dt\<\sigma(\pi_t) G_t, G_t \> \,,
\end{equation}
\begin{equation}\label{ccc}
\tq^\d (\pi)\; =\;  \sup_{G\in C^\infty_c([0,T]\times \L)} \Big\{\tq^\d_G(\pi)\Big\}\;. 
\end{equation}
  Notice that
$$
\tq^\d (\pi)=\frac1{2\d} \mc Q (\pi),
$$
where  $ \mc Q (\cdot)$   is defined   in \eqref {tm3}.
 
For a function $m$ in $\mc M$, let $m^\epsilon:\Lambda\to\bb R_+$
be given by
\begin{equation*}
m^\epsilon (u) \;=\; \frac 1{2\epsilon}
\int_{[u-\epsilon, u+\epsilon] \cap \Lambda} m(v) \, dv\;. 
\end{equation*}
When  $u\in [-1+\epsilon,1-\epsilon]$,
$m^\epsilon (u)=(m * \iota_\epsilon)(u)$, where $\iota_\epsilon$ is the approximation of the identity defined by
$$
\iota_\epsilon (u)\;=\; \frac{1}{2\epsilon}\1 \{[-\epsilon,\epsilon]\}(u)\; . 
$$
 
\begin{lemma}
\label{hs02}  There exists a positive constant $C_1$  depending only on $\rho_{\pm}$ so that   for any given  $\d_0 >0$,  
for  any $\delta$,  $0\le \delta\le \delta_0$,    for any sequence $\{\eta^N \in \mc S_N : N\ge 1\}$ 
and  for any $G\in\mc C^{\infty}_c([0,T]\times \L)$,  we have  \begin{equation*}
\limsup_{\epsilon \to 0} \limsup_{N\to\infty} \frac 1{N}  \log
\bb P^{\b,N}_{\eta^N} \Big [ 
\exp\big({\delta \, N \tq_G^{\d_0}(\pi^N  *  \iota_\epsilon)\big)} \Big] 
\; \le \; C_1( T+1)\;.
\end{equation*}
\end{lemma}

\begin{proof}
Assume without loss of generality that $\epsilon$ is small enough  so that the support of $G (\cdot, \cdot) $ is  contained in $[0,T] \times [-1+\epsilon,
1-\epsilon]$.  Let   $\theta:\overline\L\to (0,1)$  be  a smooth function such that $\theta(\mp 1)= \rho_{\mp}$.  Since $\nu^{\t(\cdot)}_N(\eta^N) \ge \exp\{- C_1' N\}$ for some finite constant $C_1'$ depending only on $\t$,  it is enough to
prove the lemma with $\bb P^{\b,N}_{\nu^{\t(\cdot)}_N}$ in place of $\bb P^{\b,N}_{\eta^N}$.

Set  $\Psi_1 (\eta) = [\eta(1)- \eta(0)]^2$   and  note  that  $\widetilde \Psi_1(a)= E^{\nu^{a}}[\Psi_1] =\sigma(a) =2 a(1-a)$,   where $\nu^{a}$ is the Bernoulli  measure with  parameter $a \in [0,1]$.
 Denote 
$B_{N,\epsilon, \d_0}^{G,\Psi_1}$    the set of trajectories $(\eta_t)_{t\in [0,T]}$  so that 
\begin{equation*}
B_{N,\epsilon,\d_0}^{G,\Psi_1} \;=\; \Big\{ \eta_\cdot \in D([0,T], {\mc S}_N) : 
\Big| \int_0^T V_{N,\epsilon}^{G^2,\Psi_1}(t,\eta_t) dt \Big| \le \frac1{\delta_0^2}\Big\} \;,
\end{equation*}
where $V_{N,\epsilon}^{G^2,\Psi_1}$ is defined in \eqref{vne}.
  By \eqref{maxlim} and  Proposition  \ref{see1},  it is enough to show 
\begin{equation*}
\limsup_{\epsilon \to 0} \limsup_{N\to\infty} 
  \frac 1 N \log
\bb P^{\b,N}_{\nu^{\t(\cdot)}_N} \Big [ 
e^{\big({\delta\, N \tq_G^{\d_0}(\pi^N  *  \iota_\epsilon)\big)}}\1\{ B_{N,\epsilon,\d_0}^{G,\Psi_1}  \} \Big] 
\; \le \; C_1 (T+1)\;.
\end{equation*}
Recalling the definition  $\tq_G^\d$, see \eqref {cccg},   we have
\begin{equation*}
\int_0^T dt\<\pi^N_t * \iota_\epsilon, \nabla G \> = 
\int_0^Tdt \sum_{x=-N+1}^{N-1} \{\eta_t(x) - \eta_t(x+1)\} \, G_t(x/N) + O_G(\ve).
\end{equation*}
Further  on  the set  $ B_{N,\epsilon,\d_0}^{G,\Psi_1}$
\begin{equation*}
\begin{aligned}
\d_0 \int_0^T dt\<\sigma(\pi^N_t * \iota_\epsilon), G^2_t \>   &\ge 
\d_0 \int_0^T dt \, \frac 1{N} \sum_{x=-N+1}^{N-1} 
G^2_t(x/N)  \, \tau_x \Psi_1 (\eta_t) \\
\ &\ \ 
- {\delta_0} O_{G^2}(N,\ve)  - \frac1{\d_0}\, ,
\end{aligned}
\end{equation*}
where $O_G(\epsilon)$ is absolutely bounded by a constant which
vanishes as $\epsilon\downarrow 0$ and $O_{G^2}(N,\ve)$ is is absolutely bounded by a constant which
vanishes as $N\uparrow\infty$.
Therefore to conclude the proof  it is enough  to show that  
 \begin{equation}\label{escc1}
 \limsup_{N\to\infty} \frac 1{N}  \log
\bb P^{\b,N}_{\nu^{\theta(\cdot)}_N}  \Big [  \exp\Big( N\, \int_0^Tdt \, V_G^\delta(t,\eta_t) \Big) \Big] 
\; \le \; C_1 T \;  
\end{equation}
for any $\d\le \d_0$, 
where  $$V_G^\delta(t,\eta) 
=   \delta\sum_{x=-N+1}^{N-1}  
 G_t(x/N) [\eta(x) - \eta(x+1) ]
-\frac {\d^2}{N} \sum_{x=-N+1}^{N-1} 
G^2_t(x/N) \, \tau_x\Psi_1(\eta). 
$$
Now, observe that $V_G^\delta=V_{\d G}^1$. 
Therefore, to prove the lemma, we need to show that for any smooth function $G$,
\begin{equation}\label{escc1b}
 \limsup_{N\to\infty} \frac 1{N}  \log
\bb P^{\b,N}_{\nu^{\theta(\cdot)}_N}  \Big [  \exp\Big( N\, \int_0^Tdt \, V_G^1 (t, \eta_t) \Big) \Big] 
\; \le \; C_1 T \;  
\end{equation}
for some constant $C_1$ that not depends on $G$.
By Feynman-Kac formula and the same arguments used in the proof of  Proposition  \ref{see1}, the expression
of the limit in the right hand side of \eqref{escc1b} is bounded above by
\begin{equation*}
\int_0^Tdt \, \sup_{f} \Big\{   \int V_G^1 (t, \eta) f^2(\eta) \nu^{\t(\cdot)}_N
(d\eta) \; + \; N \< \mc L_N f , f\>_{\nu^{\t (\cdot)}_N} \Big\}\;,
\end{equation*}
where the supremum is  over all functions $f$ in $L^2(\nu^{\t(\cdot)}_N)$
such that $\<f,f\>_{\nu^{\t(\cdot)}_N} =1$. By Lemma \ref {dirichlet}, we  replace
$N \< \mc L_N f , f\>_{\nu^{\t(\cdot)}_N}$ by $- N  (1-b) {\cd}_{0,N} \big({f},\nu_N^{\theta(\cdot)} \big) +  \frac { C_0 } b $,
where $b\in (0,1)$ is arbitrarily chosen  and $C_0$  is a   constant   depending only on $\rho_{\pm}$.  
It remains, therefore, to show that
\begin{equation}\label{escc2}
  \limsup_{N\to\infty}  \int_0^Tdt \, \sup_{f} \Big\{   \int V_G^1 (t, \eta) f^2(\eta) \nu^{\t(\cdot)}_N
(d\eta) \; - N  (1-b) {\cd}_{0,N} \big({f},\nu_N^{\theta(\cdot)} \big)   \Big\}\
\; \le \;  C_1  T. 
\end{equation}
 We  split 
$$\int V_G^1 (t, \eta) f^2(\eta) \nu^{\t(\cdot)}_N
(d\eta)= I_1-I_2, $$  
where 
$$ I_1=\sum_{x=-N+1}^{N-1} G_t(x/N) \int  \{\eta (x) - \eta (x+1)\} 
\, f^2(\eta)  \, d\nu_N^{\t(\cdot)} (\eta), $$
$$ I_2=\frac {1}{N} \sum_{x=-N+1}^{N-1} G^2_t(x/N) \int     f^2(\eta) \, \tau_x\Psi_1(\eta) \, d\nu_N^{\t(\cdot)} (\eta). $$
We   estimate  $I_1$ in term of $I_2$ and ${\cd}_{0,N} \big({f},\nu_N^{\theta(\cdot)}) $.
By changing  variables  $\eta' = \eta^{x,x+1}$,  we have that 
\begin{equation}
\label{mo05}
\begin{split} & I_1=    
 \frac{1}{2} \sum_{x=-N+1}^{N-1} G_t(x/N) \int  \{\eta (x) - \eta (x+1)\} 
\, \{ f^2(\eta)  - f^2(\eta^{x,x+1})\} \, d\nu_N^{\t(\cdot)} (\eta) \\
& \quad
+\frac{1}{2} \; \sum_{x=-N+1}^{N-1} G_t(x/N)\,  \int  \{\eta (x) - \eta (x+1)\} \, R_N(x,x+1;\t,\eta)
\, f^2(\eta)   \, d\nu_N^{\t(\cdot)} (\eta)\;,
\end{split}
\end{equation}
where
$
R_N(x,x+1;\t,\eta)
$
is defined in \eqref{reste}. 
For  the first term of \eqref {mo05}, by inequality \eqref{sch-ineq} and Taylor expansion,   we have
\begin{equation} \label {ett5}
\begin{aligned} & \frac{1}{2} \sum_{x=-N+1}^{N-1} G_t(x/N) \int  \{\eta (x) - \eta (x+1)\} 
\, \{ f^2(\eta)  - f^2(\eta^{x,x+1})\} \, d\nu_N^{\t(\cdot)} (\eta)\cr & \le 
 \frac{aN}{4} {\cd}_{0,N} \big({f},\nu_N^{\theta(\cdot)} \big)\\
&\qquad
 +\,
\frac{1}{4aN}\sum_{x=-N+1}^{N-1} G^2_t(x/N)   \int   \tau_x\Psi_1(\eta) \,\big[ f(\eta)+ f(\eta^{x,x+1})\big]^2  
\, d\nu_N^{\t(\cdot)} (\eta)\\
& \ \ \le \frac{aN}{4}{\cd}_{0,N} \big({f},\nu_N^{\theta(\cdot)} \big) \, +\,   \frac1{aN}\, C(G) \\
&\qquad
\; +\; 
\frac{1}{aN} \sum_{x=-N+1}^{N-1}  G^2_t(x/N) \int     \tau_x\Psi_1(\eta) \,f^2(\eta) 
\, d\nu_N^{\t(\cdot)} (\eta)\;
\end{aligned}
\end{equation}
where $C(G)$ is some constant that depends on $G$.
For the second term of  \eqref {mo05},  by   \eqref{sch-ineq} and  Taylor expanding $R_N$   we have that
\begin{equation}
\label{et4} \begin{aligned} & \Big| \frac{1}{2} \; \sum_{x=-N+1}^{N-2} G_t(x/N)\,  \int  \{\eta (x) - \eta (x+1)\} \, R_N(x,x+1;\t,\eta)
\, f^2(\eta)   \, d\nu_N^{\t(\cdot)} (\eta)\Big| \cr & \le 
{C\, a} + \frac{1}{N\, a}\sum_{x=-N+1}^{N-1} \int G_t(x/N)^2 \,  \tau_x\Psi_1(\eta) \, f^2(\eta)   \, d\nu_N^{\t(\cdot)} (\eta)\;,\cr & =  {C\, a} + \frac{1}{ a} I_2
\end{aligned}
\end{equation}
for all $a>0$, for some positive constant $C$  depending only on $\rho_{\pm}$.     
Taking into account \eqref {mo05}, \eqref {et4} and \eqref {ett5}  we have
\begin{equation} \label {ett6}      I_1  \le    \frac{2}{ a} I_2+ 
\frac{aN}{4}{\cd}_{0,N} \big({f},\nu_N^{\theta(\cdot)} \big)   +  {C\, a}  + \frac1{aN}\, C(G)  
\end{equation}
We conclude the proof, by taking  $a = {2}$ and $b= \frac 12  $ in \eqref  {escc2}.\end{proof}

 The following corollary allows to  show  Lemma \ref{lem3}.  

\begin{corollary}
\label{cor03} 
Fix a sequence $\{G_j : j\ge 1\}\subset \mc C^{\infty}_c([0,T]\times \L)$,   
 $\delta_0 >0$ and a sequence $\{\eta^N \in {\mc S}_N : N\ge 1\}$ of configurations.
There exists a positive constant $C_1$  depending only on the values  $\rho_{\mp}$,
such that for any $0<\delta\le \d_0$ and any $k\ge 1$
\begin{equation}\label{aaa}
\limsup_{\epsilon \to 0} \limsup_{N\to\infty} \frac 1{N}  \log
\bb P^{\b,N}_{\eta^N} \Big [ 
\exp\big({\delta \, N \sup_{1\le j\le k}\tq_{G_j}^{\d_0}(\pi^N  *  \iota_\epsilon)\big)} \Big] 
\; \le \; C_1 (T+1)\;.
\end{equation}
\end{corollary}

\begin{proof}
{}From \eqref{maxlim}, the limit in \eqref{aaa} is bounded above by
$$
\max_{1\le j\le k}\Big\{
\limsup_{\epsilon \to 0} \limsup_{N\to\infty} \frac 1{N}  \log
\bb P^{\b,N}_{\eta^N} \Big [ 
\exp\big({\delta \, \, N \tq_{G_j}^{\d_0}(\pi^N  *  \iota_\epsilon)\big)} \Big]\Big\}.   
$$
By  Lemma \ref{hs02} the thesis follows. 
\end{proof}

\section{Hydrodynamic and hydrostatic limits}
\label{hydro}
We prove in this section the hydrodynamic and  hydrostatic limit for our system. The proof is based on the method introduced
in \cite{gpv}  
 for  the hydrodynamic limit  and in \cite{flm}  for hydrostatic, 
  taking   into account, as explained in the introduction,   the lack of comparison and  maximum principle of \eqref {eq:1}.
\subsection{The steps to prove Theorem \ref{th-hy}}

Following \cite{gpv} we divide the proof of the hydrodynamic behavior in
three steps: tightness of the measures $(Q_{\mu_N}^{\b,N}) $, an energy estimate to provide  the needed regularity for 
functions in the support of any limit point of the sequence $(Q_{\mu_N}^{\b,N})$, 
and   identification of the  support of limit point of the sequence $(Q_{\mu_N}^{\b,N})$ as weak
solution of \eqref{eq:1}.  We then refer to \cite{kov}, Chapter 4, that present
arguments, by now standard, to deduce the hydrodynamic behavior of the empirical measures from the preceding 
results and the uniqueness of the weak solution to equation \eqref{eq:1}. 
 
\begin{lemma} {\bf (Tightness)}
The sequence $(Q_{\mu_N}^{\b,N})$  is tight and all its limit points
$Q^{\beta,*}$ are concentrated on   the following  set 
\begin{equation}
\label{abscont}
Q^{\beta,*}\Big\{ \pi \, :\, 0\le \pi_t(u)\le 1, \quad  t \in [0,T], \quad u \in [-1,1]  \Big\}=1\; .  
\end{equation}
\end{lemma} 

We then  show that  $Q^{\beta,*}$ is supported on densities $\r$ that satisfy \eqref{eq:1} in the weak sense. 

We start defining  for 
  $G \in {\cc_0}^{1,2}([0,T]\times \L)$  and   $\e>0$   
\begin{equation}
\label{id-lim}
\begin{aligned} 
&  \cb^{G,N}_{\varepsilon}=
\int_\L G_T(u) \pi^N(\eta_T)(u)du -\int_\L G_0(u) \pi^N(\eta_0)(u)du\\
&\quad
- \int_0^T \int_\L \partial_s G_s(u) \pi^N(\eta_s)(u)duds  
- \int_0^T \int_\L \Delta G_s(u) \pi^N(\eta_s) (u)duds \\
&\quad
-\frac \b  2\frac1{N}\sum_{x\in\L_N}  \int_0^T   
\big(\nabla G_s\big)(x/N) \Big\{ \sigma\big( \eta_s^{[\varepsilon N]}(x)\big)
\nabla^N\big( \jn\star \pi^N(\eta_s) \big) (x/N)\Big\}ds\\
&\quad
+\int_0^T dt
\left [\rho_+(\nabla G_t)(+1) - \rho_-(\nabla G_t)(-1)  \right] \; ,
\end{aligned}
\end{equation}
where $\eta_s^{[\varepsilon N]}(x)$ is the  local mean defined in \eqref {average} and     
$\nabla^N G_s(x/N)$ stands for  the discrete gradient  of  $G_s(x/N)$ defined in  \eqref{derive1}. 
 
\begin{proposition}
\label{lem2}
{\bf (Identification of the limit
equation).}
For  
any  function $G$ in ${\cc_0}^{1,2}([0,T]\times \L)$  and any $\d>0$ 
we have 
\begin{equation}
\label{ident-lim}
\limsup_{\varepsilon \to 0} \limsup_{\N \to \infty} 
\Pb_{\mu_N}^{\b,N} \left (
\left | \cb^{G,N}_{\varepsilon} \right| \ge\d \right )=0 \;.
\end{equation}
\end{proposition}
 The last statement is an energy estimate. Every limit point $Q^{\beta,*}$ of the sequence 
$(Q_{\mu_N}^{\b,N})$ is concentrated on paths whose densities $\r \in L^2\big(0,T; H^1(\L)) $. 
 
\begin{lemma}
\label{lem3} {\bf (Energy estimate)} Let
$Q^{\beta,*}$  be a limit point of the sequence $(Q_{\mu_N}^{\b,N})$. Then, 
\begin{equation}
\label{enest} 
{Q^{\beta,*}} \Big[ L^2\big(0,T; H^1(\L) \big)\Big]=1\; .
\end{equation}
\end{lemma}

\subsection{Proof of Proposition \ref{lem2}}
Let  $Q^{\beta,*}$  be  a limit point of the sequence 
$(Q_{\mu^N}^{\b,N})$ and assume, without loss of
generality, that  $Q_{\mu^N}^{\b,N}$ converges to $Q^{\beta,*}$.
Fix a function $G$ in ${\cc}_0^{1,2}([0,T]\times\L)$. Consider the   $\Pb_{\mu^N}^{\b,N}$ 
martingales  with respect to the natural filtration associated with $(\eta_t)_{t\in[0,T]}$,
$M_t^G\equiv M_t^{G,N,\b}$ and  $\cn_t^G \equiv \cn_t^{G,N,\b}$, 
$t\in[0,T]$, defined by 
\begin{equation}
\label{lMMa1}
\begin{aligned}
&M_t^G =<\pi_t^N,G_t>-<\pi_0^N,G_0> -\int_0^t \big (  <\pi_s^N,
\partial_s G_s>+N^2
\cl_N^{\b} <\pi_s^N, G_s> \big)   d s \, , \\
& \cn_t^G \; =\; \left(M_t^G \right)^2 \\
&\qquad
 -\; \int_0^t \left\{ N^2 \cl_N^{\b}
\big(<\pi_s^N,G_s>\big)^2 - 2 <\pi_s^N,G_s> N^2 \cl_N^{\b}
<\pi_s^N,G_s>\right\} d s \; .
\end{aligned}
\end{equation} 
A computation of the integral term of $\cn_t^G$ shows
that the expectation of the quadratic variation of 
$M_t^G$ vanishes as $N\uparrow 0$. Therefore, by Doob's
inequality, for every $\d >0$, 
\begin{equation}
\label{bv}
\lim_{N\rar \infty} \Pb_{\mu_N}^{\b,N} \Big[ \sup_{0\le t\le T}
|M_t^G|>\d \Big] \; =\;0 \; . 
\end{equation}
 Since for any $s\in[0,T]$ the  function
$G_s$ vanishes at the boundary of $\L$, a summation by
parts permits to rewrite the integral term of $M_t^G$ as
\begin{eqnarray*}
&&\int_0^t <\pi_s^N,\partial_s G_s> d s \\
&&
\; -\; \int_0^t  \, 
N\Big\{ \sum_{x=-N+1}^{N-1} \big(\nabla^N
G_s \big) (x/N)  C_N^\b (x,x+1,\eta_s)\big( \nabla_{x,x+1}\eta_s (x)\big)\Bigr\}ds \; ,
\end{eqnarray*}
where $\nabla^N$ is defined in \eqref{derive1}.

{}From Lemma \ref{b1}, a summation by parts and Taylor expansion permit to rewrite the last expression as
\begin{equation*}
\begin{aligned}
&O(N)\; +\;
\int_0^t <\pi_s^N,\partial_s G_s> d s
+\int_0^t <\pi_s^N,\Delta G_s> d s \\
&
+\int_0^t \Big\{-\nabla G_s(1) \eta_s(N)\, + \,\nabla G_s(-1) \eta_s(-N) \Big\}ds\\
&
+\frac{\b}{2N}\int_0^t\Big\{ \sum_{x\in \L_N} \big(\nabla
G_s \big) (x/N) \big(\nabla_{x,x+1}\eta_s (x) \big)^2\;  \nabla^N( \jn\star \pi^N(\eta_s))(x/N)
\Big\}ds \, .
\end{aligned}
\end{equation*}
Next, we use  the replacement lemma stated in 
Proposition \ref{see1} and Proposition \ref{seeb}. We   obtain that the integral term of the martingal $M_t^G $ can
be replaced by
\begin{equation*}
\begin{aligned}
&
\int_0^t <\pi_s^N,\partial_s G_s> d s
+\int_0^t <\pi_s^N,\Delta G_s> d s \\
&\quad
+\int_0^t \Big\{-\nabla G_s(1) \rho_+\, + \,\nabla G_s(-1) \rho_- \Big\}ds\\
&\quad
+\frac{\beta}2 \int_0^t\frac{1}{N}\Big\{ \sum_{x\in \L_N} \big(\nabla
G_s \big) (x/N)\sigma \left(\eta^{\varepsilon N}(x)   \right)  
 \nabla^N( \jn \star \pi^N(\eta_s))(x/N)
\Big\}ds \; .
\end{aligned}
\end{equation*}
This concludes the proof of the lemma.
\qed

\subsection{Steps to prove Theorem \ref{th-hy1}}

 Let   $\mu^{stat}_N=     \mu^{stat}_N(\b,\rho_-, \rho_+)$ be    the unique  stationary  measure of the irreducible Markov process $(\eta_t)_{t\ge 0}$ with generator 
$\mc L_N $.    
 From Tchebyshev's unequality, we need to show that 
\begin{equation}\label{ss1a}
\lim_{N\to\infty} E^{\mu^{stat}_{N}}\Big [ \Big|\big< \pi^N,G \big>- \big< \bro,G\big>
\Big|\Big]=0\, ,
\end{equation}
where $E^{\mu^{stat}_{N}}$ stands for the expectation with respect to the stationary measure $\mu^{stat}_N$.
It is enough to prove that any subsequence of the sequence of real numbers in the limit \eqref{ss1a} vanishes. Without loss of generality we
consider a sequence in  \eqref{ss1a} as a subsequence along which the limit exists. 

\medskip 
 Denote by   $Q^{\beta,N,stat}:= Q^{\beta,N}_{\mu^{stat}_N}$   the probability measure on the Skorohod space
$D\big([0,T]; \mc M\big) $ induced by the Markov process $(\pi_t^N)\equiv (\pi_N(\eta_t))$, when the initial measure is $\mu^{stat}_N$.

By  the first part of Theorem \ref{th-hy} we have  that all limit points of the sequence
$Q^{\beta,N,stat}$ are concentrated on  $\mc A_{[0, T]}$ for any $T>0$, i.e  all its  limit points  are concentrated on the weak solutions of the hydrodynamic equation for  some unknown initial profile.

 Let   $(Q^{\b,N_k,stat} )$ be a
sub-sequence converging to  a limit point which we denote by $Q^{\beta,*,stat}$.  Note that  different subsequences might have different limit points.  Let $\beta$  small enough and denote by $\bar \rho$ the unique   stationary solution   of \eqref {eq:1}, see Theorem \ref {stat2}. 
   By stationarity  we have for any $\d>0$, 
\begin{equation*} 
E^{\mu^{stat}_{N_k}}\Big [ \Big|\big< \pi^N,G \big>- \big< \bro,G\big>
\Big|\Big]
  \; =\;   \E^{\beta,N_k}_{\mu_{N_k}^{stat}}  
\Big [  \Big|\big< \pi_T^N,G \big>- \big< \bro,G\big>
\Big|\Big] .
\end{equation*} 
Since the integrand   is bounded   
  we have the following: 
\begin{equation}
\begin {split}  
 \lim_{k\to \infty} &
  \E^{\beta,N_k}_{\mu_{N_k}^{stat}}  
 \Big\{     \Big|\big< \pi_T^N,G \big>- \big< \bro,G\big>
\Big|\Big\}\cr
\  & \; =\; 
 E^{  Q^{\beta,*,stat}}
\Big\{   \left ( \Big|\big< \rho_T ,G\big>- \big< \bro,G\big>
\Big| \1_{\{\mc A_{[0, T]}\}} \big(\rho\big)  \right ) \Big\}\cr
 &\ \ \; \le \; \| G\|_2
\ E^{  Q^{\beta,*,stat}}
 \Big\{   \big\| \rho_T  -\bro \big\|_2
\1_{ \{\mc A_{[0, T]}\} }    \big(\rho\big)      \Big\} \cr & \le 
 \| G\|_2  e^{-c(\beta)T}
\end {split}
\end {equation}
by Theorem \ref {stat2}  and    $ \|v\|_2$ denotes the $L^2(\Lambda)$ norm of $v$.
 Then letting $ T \to \infty$ we show the thesis.  
 \qed
  \section{Large deviations}
\label{large-dev}
In this section we prove some properties of the rate function and we present the main steps to derive the
large deviations results.

Let $L^2(\L)$ be the Hilbert space of functions $G:\L
\to \R$ such that $\int_\L | G(u) |^2 du <\infty$ equipped with
the inner product
\begin{equation*}
\<G,J\>_2 =\int_\L G(u) \,  J (u) \, du\; .
\end{equation*}
The norm of $L^2(\L)$ is denoted by $\|
\cdot \|_2$.

Let $H^1(\L)$ be the Sobolev space of functions $G$ with
generalized derivatives $\nabla G$	
in $L^2(\L)$. $H^1(\L)$ endowed with the scalar product
$\<\cdot, \cdot\>_{H^1}$, defined by
\begin{equation*}
\<G,J\>_{H^1} = \< G, J \>_2 + 
\<\nabla G \, , \, \nabla J \>_2\;,
\end{equation*}
is a Hilbert space. The corresponding norm is denoted by
$\|\cdot\|_{H^1}$.
  Denote by  $H^{1}_0(\Lambda)$ the closure of $C^\infty_c(\L)$ (the set of infinitely differentiable functions   from $\L$ to $\R$   with compact support in $\L$) in $H^1(\L)$.  
Denote by $H^{-1}(\Lambda)$ the Hilbert space, dual of   $H^{1}_0(\Lambda)$, equipped 
with the norm $\Vert\cdot\Vert_{-1}$
 \begin{equation*}
\Vert v\Vert^2_{-1} = \sup_{G\in\mc C^{\infty}_c(\Lambda)}
\left\{2\langle v,G\rangle_{-1,1} -
\int_{\Lambda} | \nabla G(u)|^2du \right\}\, , 
\end{equation*}
where $\langle v,G\rangle_{-1,1}$ stands for the  duality between $H^1_0$ and $H^{-1}$.
Fix $T>0$.
For a Banach space $(\bb
B,\Vert\cdot\Vert_{\bb B})$ we denote by $L^2([0,T],\bb B)$
the Banach space of measurable functions $U:[0,T]\to\bb B$ for which
\begin{equation*}
\Vert U\Vert^2_{L^2([0,T],\bb B)} \;=\; 
\int_0^T\Vert U_t\Vert_{\bb B}^2\, dt \;<\; \infty
\end{equation*}
holds. 

\subsection{Properties of the rate function}

Denote 
 $$\cb_\gamma^{\rho_\pm} = \{ \pi \in D([0,T], \mc M): \pi_0(\cdot) = \gamma (\cdot); \quad \pi_t(\pm 1) = \rho_\pm,   t\in (0,T] \}. $$

 \begin{lemma}
\label{lem01}
Let $\pi$ be a trajectory in $D([0,T],\mc M)$ such that
$\hat I_T^\b(\pi|\gamma)<\infty$. Then $\pi$ belongs to $\cb_\gamma^{\rho_\pm}$.
\end{lemma}
The proof   is similar to the one of Lemma 3.5 in \cite{bdgjl3}.
To prove the lower-semicontinuity of the rate function, we need the next results
 \begin{lemma} 
\label{g05} For any $\b\ge 0$,
there exists a constant $C_0=C_0(\b)$ such that
\begin{equation*}
\int_0^T\| \partial_t \pi_t\|_{-1}^2  
\;\le\; C_0 \Big\{ I_T^\b (\pi |\gamma ) +  \cq(\pi)\Big\}
\;,\quad \mc Q (\pi) \;\le\;  C_0\big\{ 1 + I_T^\b (\pi |\gamma )\big\}
\end{equation*}
for all $\pi$ in $D([0,T], \mc M)$.  
\end{lemma}

\begin{proof}
The proof is the same
as in Proposition 4.3. \cite{qrv} or Theorem 3.3. in \cite{mo1}, or Lemma 4.9. in \cite{blm}  
\end{proof}

\begin{lemma}
\label{g06} 
Let $\{\rho^n : n\ge 1\}$ be a sequence of functions in $L^2
([0,T]\times\L)$ such that
$$
\int_0^T dt \, \Vert \rho^n_t \Vert^2_{H^1}   \;+\; 
\int_0^T dt \, \|\partial_t \rho^n_t \|_{-1}^2  
\; \le \; C_0
$$
for some finite constant $C_0$ and all $n\ge 1$. Suppose that the
sequence $\rho^n$ converges weakly in $L^2 ([0,T]\times [-1,1])$ to some $\rho$.
Then, $\rho^n$ converges strongly in $L^2 ([0,T]\times [-1,1])$ to $\rho$.
\end{lemma}

\begin{proof}
Recall that $H^1 (\L) \subset L^2(\L)\subset H^{-1}(\L)$.
By \cite[Theorem 21.A]{z}, the embedding $H^1 (\L) \subset
L^2(\L)$ is compact. Hence, by \cite[Lemma 4, Theorem 5]{Si},
the sequence $\{\rho^n : n\ge 1\}$ is relatively compact in $L^2(0,T;
L^2(\L))$.  In particular, weak convergence of the sequence
$\{\rho^n : n\ge 1\}$ implies strong convergence.
\end{proof}																																																																																																																																																																										²

\begin{theorem}
\label{th4}
The functional $I_T^\b(\cdot|\gamma)$ is lower semicontinuous and has
compact level sets. 
\end{theorem}

\begin{proof} Theorem \ref{th4} is   proven applying Lemma \ref{g05} and Lemma \ref{g06}.
See Theorem 3.4. in \cite{mo1} or Lemma 4.2. in \cite{blm}. 
\end{proof}
\medskip \noindent
We provide an explicit representation for the rate
function $I_T^\b(\cdot |\gamma)$ when it is finite.  For  $\pi \in D([0,T], \mc M)$, denote by $H^1_0(\sigma(\pi))$ the Hilbert space
induced by $\mc C^{1,2}_0([0,T]\times [-1,1])$ endowed with the inner
product $\langle\cdot,\cdot\rangle_{\sigma(\pi)}$ defined by
\begin{equation*}
\langle F,G\rangle_{\sigma(\pi)}
=\int_0^Tdt\; \langle\sigma(\pi_t),\nabla F_t\cdot\nabla G_t\rangle
\,. 
\end{equation*}

Induced means that we first declare two functions $F,G$ in $\mc
C^{1,2}_0([0,T]\times [-1,1])$ to be equivalent if $\langle
F-G,F-G\rangle_{\sigma(\pi)} = 0$ and then we complete the quotient
space with respect to the inner product
$\langle\cdot,\cdot\rangle_{\sigma(\pi)}$. The norm of
$H^1_0(\sigma(\pi))$ is denoted by $\Vert\cdot\Vert_{\sigma(\pi)}$.

\begin{lemma}
\label{lem05}
Take $\pi \in D([0,T],\mc M)$ with $I_T^\b(\pi |\gamma)<\infty$. Then,  it is uniquely  determined a  function
$F$ in $H^1_0(\sigma(\pi))$ such that $\pi$ is the  weak solution of  the following
 boundary value  problem:
\begin{equation}
\label{f05}
\begin{cases}
\partial_t\pi  & = \;\;  \Delta\pi -
\nabla \cdot \big\{ \sigma (\pi) \big[\frac{\b}2 \nabla (\jn\star \pi)  +\nabla F \big]\big\} \, \quad \hbox {in} \quad \Lambda \times (0,T),\\ 
\pi_0 (\cdot)    & =  \;\; \gamma (\cdot)\, \quad \hbox {in} \quad \Lambda,\\
\pi_t (\pm 1)   & = \;\; \rho_\pm \;\;\; 
\hbox{ for }\; 0\leq t\leq T \, .
\end{cases}
\end{equation}

Moreover,
\begin{equation}
\label{f06}
I_T^\b(\pi|\gamma) = \frac{1}{2}\Vert F\Vert_{\sigma(\pi)}^2
=\frac12 \int_0^Tdt\; \langle\sigma(\pi_t)\nabla F_t\cdot\nabla F_t\rangle\, .
\end{equation}
\end{lemma}
\begin {proof} 
By assumption $I_T^\b(\pi|\gamma) = \hat I_T^\b(\pi|\gamma)<\infty$, defined in \eqref {2:Ib}. 
Then one   proceeds as   in \cite {kov}   with the only difference that  
because    the boundary conditions  the space is $H^1_0(\sigma(\pi))$. 
  \end {proof} 
 \medskip 
 
\begin{lemma}
\label{enT1}
Let $\rho  \in L^2 ([0,T], H^1 (\Lambda))$ be  the weak solution of the boundary value problem
 \eqref {eq:1}  then
 $$  I_T^\b(\rho|\gamma)= \hat I_T^\b(\rho|\gamma)=0, \quad \hbox {and}  \quad   \cq(\rho) < \infty. $$
Further if $  I_T^\b(\rho|\gamma)=0$,  then  $\rho  \in L^2 ([0,T], H^1 (\Lambda))$ is the weak solution of the boundary value problem
 \eqref {eq:1}.
\end{lemma}
\begin {proof}
We start showing that   if $\rho  \in L^2 ([0,T], H^1 (\Lambda))$ is   the weak solution of the boundary value problem
 \eqref {eq:1}  then   $\cq(\rho) < \infty$.
Take  $ F(\rho)=   \rho \log \rho +(1-\rho)\log (1-\rho) $, for  $\rho \in [0,1]$.   Since
 $ \int_{\Lambda}   F(\rho _t(u)) du $ is a bounded quantity for all $ t \in R^+$, 
 we have that
 $$ \int_0^T  dt  \frac {d} {dt}  \int_{\Lambda}    F(\rho_t (u)) du  =   \int_{\Lambda}    [ F(\rho_T (u)) -F(\rho_0 (u)) ]du .$$ 
 Notice that 
$$   F'(\rho)= \log \frac \rho {(1-\rho)} \quad  \hbox {and}  \quad    F''(\rho)=\frac 1 {\rho (1-\rho)}= \frac 1 {\chi (\rho)} $$
are not   uniformly bounded for $\r \in (0,1)$.Therefore we need some care  to  derive $ F(\rho_t (u))$ with respect to $t$.
We   consider a sequence of smooth functions  
$$F_n(\rho)= \big(1+\frac2n\big)^{-1}(\rho+\frac 1 n) \log  (\rho+\frac 1 n) +(1-  \rho+\frac 1 n)\log (1- \rho+\frac 1 n)  $$  
so that
 $ \lim_{n \to \infty} F_n (a) =F (a)$.
  We have 
 \begin {equation} \label {entt1}   \int_0^T  dt  \frac {d} {dt}  \int_{\Lambda}   F_n(\rho_t (u)) du  
   =  \int_0^T  dt   \int_{\Lambda}   F'_n(\rho_t (u))  \frac {d} {dt}  \rho_t (u).
  \end {equation}
 To avoid boundary terms,  take a   smooth  function 
  $b(\cdot)$ defined on a neighborhood of $[-1,1]$  such that $b(\mp 1)=\rho^\mp$  and $ 0< \rho^- \le  b(\cdot) \le \rho^+ <1 $.  Denote  $$U_n(t,u)= F_n'(\rho_t(u))- F_n'(b(u)).$$ 
We have 
 \begin {equation} \label {entt1a} \begin {split} &
  \int_0^T  dt   \int_{\Lambda}   F'_n(\rho_t (u))  \frac {d} {dt}  \rho_t (u) \\
& \qquad
=  \int_0^T  dt   \int_{\Lambda}   U_n(\rho_t (u))  \frac {d} {dt}  \rho_t (u) +   \int_0^T  dt   \int_{\Lambda}   F'_n(b(u))  \frac {d} {dt}  \rho_t (u)  \cr 
&\qquad
=\int_0^T  dt   \int_{\Lambda}   U_n(\rho_t (u))  \frac {d} {dt}  \rho_t (u) du+    \int_{\Lambda}   F'_n(b(u))   [\rho_T (u)- \rho_0 (u)]du. 
 \end  {split}  \end {equation}
  Taking into account  \eqref  {entt1}, \eqref   {entt1a}  and  $U_n\in L^2([0,T], H_0^1)$, we get
\begin{equation} \label {entt2} 
\begin {split} & \int_0^T  dt  \frac {d} {dt}  \int_{\Lambda}   F_n(\rho_t (u)) du -
\int_\L F_n'(b(u))\big[ \rho_T(u) -\rho_)(u)\big] du \, 
\cr &= - \int_0^T  dt   \int_{\Lambda}  \nabla U_n (t,u) \left [ \nabla \rho_t (u) - \beta  \rho_t (u)(1-\rho_t(u)) ( \jn \star \nabla \rho_t ) (u) \right ]\, .
\end  {split}  \end {equation}
Denote $\chi_n(a)=(a+\frac1n) (1+\frac1n -a)$. We have that 
$$  \nabla U_n (t,u)=  \frac  { \nabla \rho_t (u) } { \chi_n (\rho_t (u) )}   -   \frac  { \nabla b(u) }  { \chi_n (b(u))}.$$
 Taking this into account and collecting the above estimates, we obtain  
\begin{equation}
 \begin{aligned}
 & \int_0^T dt\int_\L \frac{(\nabla \rho_t(u))^2}{\chi_n(\rho_t(u))}du\\
&\quad \le -\int_{\Lambda}    [ F(\rho_T (u)) -F(\rho_0 (u)) ]du +
\int_\L F_n'(b(u))\big[ \rho_T(u) -\rho_0(u)\big] du \\
& \qquad +\int_0^T dt\int_\L \frac{\nabla b(u) \cdot \nabla\rho_t(u)}{\chi_n(b(u))} du\\
&\qquad +\b \int_0^T dt\int_\L \nabla\rho_t(u))\frac{\chi(\rho_t(u))}{ \chi_n(\rho_t(u))} (\jn*\nabla \rho_t)(u)du\\
&\qquad  -\b \int_0^T dt\int_\L \nabla b(u)\frac{\chi(\rho_t(u))}{ \chi_n(b(u))}(\jn* \nabla \rho_t)(u) du\, .
 \end{aligned}
\end{equation}
Since $b(\cdot)$ is bounded below   by a strictly positive constant and
above by a constant  strictly smaller than 1, and since
$$ \int_0^T  dt   \int_{\Lambda}  du  \nabla \rho_t (u) (\jn \star \nabla \rho_t ) (u) \le C $$
for some constant $C$, we obtain, uniformly in $n$ 
$$
\int_0^T dt\int_\L \frac{(\nabla \rho_t(u))^2}{\chi_n(\rho_t(u))}du \le C'
$$
for some finite constant $C'$ which depends only on $b$ and $T$.  To
conclude the proof it remains to apply Fatou's Lemma and recall the definition of   $\cq(\rho)$ given in \eqref {tm4}.
We have shown that   $\cq(\rho) < \infty$. 
  By  Lemma \ref {lem05}    we  conclude that $  I_T^\b(\rho|\gamma)=0$.
Similar arguments allow to prove the second statement of the lemma. 
\end {proof} 

\subsection{  Comparison  between $\hat I_T^\b(\cdot|\gamma)$ and $ \hat I_T^0(\cdot|\gamma)$ } 
 Next, we compare the rate functional $\hat I^\b_T(\cdot|\g)$ with the
rate functional $\hat I^0_T(\cdot|\g)$  of the symmetric simple exclusion process (i.e. $\b=0$).

\begin{lemma}
\label{compar-I}
For  $\pi \in D([0,T],\mc M)$, with finite energy   $\mc Q(\pi)<\infty$,    we have  
\begin{equation}
\label{eqcompar}
\begin{aligned}
&\frac12 \hat I_T^0(\pi|\gamma)\, -\, \frac{\b^2}{16}\int_0^T dt\int_\L \big(\nabla\pi_t\big)^2
\le\;  \hat I_T^\b(\pi|\gamma) \\
&\qquad\qquad\qquad\qquad
 \ \le \; 2 \hat I_T^0(\pi|\gamma) \, +\, \frac{\b^2}{8}\int_0^T dt \int_\L \big(\nabla\pi_t\big)^2\,\; .
\end{aligned}
\end{equation}
\end{lemma}

\begin{proof}
Fix $\pi\in D([0,T], \mc M)$ with finite energy and $G\in C^{1,2}_0([0,T]\times [-1,1])$.
 Recall from \eqref{Jb1}, \eqref{lb1} and \eqref{2:Ib} the definitions of  $\cj_G^\b(\pi)$, $\ell_G^\b$    and $I_T^0(\pi)$. By  the inequality 
$ab \le \frac 12 a^2+ \frac 12 b^2$ we obtain 
\begin{equation*}
\begin{aligned}
&\left|\ell_G^\b(\pi,\g)-\ell_G^0(\pi,\g) \right|\,=\, 
\left| \frac\beta2\int_0^T dt \int_\L \sigma(\pi_t)(\nabla G_t)\cdot\nabla(\jn\star \pi_t)\right|\\
&\qquad \qquad\qquad
 \le \frac14\int_0^T dt \int_\L \sigma(\pi_t)(\nabla G_t)^2\, +\,
\frac{\beta^2}4\int_0^T dt \int_\L \sigma(\pi_t)[\nabla(\jn\star \pi_t)]^2\, .
\end{aligned}
\end{equation*}
Since for each $u$, $\jn(u,v)dv$ is a probability density on $\L$  and $ \s (\cdot)  \le 1/2$, 
by Lemma \ref{lem-jn}, Jensen inequality 
and Fubini's Theorem, 
\begin{equation*}
\begin{aligned}
&\frac{\beta^2}4\int_0^T dt \int_\L \sigma(\pi_t)[\nabla(\jn\star \pi_t)]^2 \\
&\qquad \le 
\frac{\beta^2}8\int_0^T dt \int_\L \jn\star (\nabla\pi_t)^2
\,=\,\frac{\beta^2}8\int_0^T dt \int_\L(\nabla\pi_t)^2\; .
\end{aligned}
\end{equation*}
Hence
\begin{equation*}
\begin{aligned}
&\cj_G^\b(\pi) \,\le\,
\ell_G^0(\pi) \, -\, \frac14 \int_0^T dt \big< \sigma(\pi_t), (\nabla G_t)^2\big>
\,+\,\frac{\beta^2}8\int_0^T dt \int_\L(\nabla\pi_t)^2\\
& \quad \ \le \frac12 \sup_{G\in C^{1,2}_0([0,T]\times \L)}
   \Big\{ 2\ell_G^0(\pi) -\, \frac12 \int_0^T dt \big<\sigma(\pi_t), (\nabla G_t)^2\big> \Big\}\\
& \quad \quad 
\,+\,\frac{\beta^2}8\int_0^T dt \int_\L(\nabla\pi_t)^2\\
&\quad =2\widehat I_T^0(\pi)\,+\,\frac{\beta^2}8\int_0^T dt \int_\L(\nabla\pi_t)^2\; .
\end{aligned}
\end{equation*}
Now, it is enough to take the supremum over $G\in C^{1,2}_0([0,T]\times [-1,1])$ to obtain
$$
\widehat I_T^\b(\pi) \,\le\,2\widehat I_T^0(\pi)\,+\,\frac{\beta^2}8\int_0^T dt \int_\L(\nabla\pi_t)^2\; .
$$

The inequality in the left hand side of the statement is obtained in the same way.
\end{proof}

 Setting  $\b=0$ in the boundary  value problem \eqref  {eq:1} one gets the following 
 boundary  value problem for the  heat equation: 

\begin{equation}
\label{f00}
\begin{cases}
\partial_t\rho  & = \;\;  \Delta\rho \, \qquad \hbox {in } \qquad  \Lambda \times (0,T) ,\\ 
\rho_0 (\cdot)    & =  \;\; \gamma (\cdot)  \,  \qquad \hbox {in } \qquad  \Lambda ,\\
\rho_t (\pm 1)|  & = \;\; \rho_\pm \;\;\; 
\hbox{ for }\; 0\leq t\leq T \, .
\end{cases}
\end{equation}

\begin{lemma}\label{lem4.7} Let $\rz$ be the solution of \eqref{f00}, we have
\begin{equation}
\label{b=0}
\begin{aligned}
\hat I_T^\b(\rz|\gamma)\, & \le\, 
 \frac{\b^2}8 \int_0^T dt\big\<\s(\rz_t),\jn\star  (\nabla\rz_t)^2\big\> \\
\ & \le \frac{\b^2}{ 16}\int_0^T dt\int_\L\big|\nabla\rz_t\big|^2\; .
\end{aligned}
\end{equation}
\end{lemma}

\begin{proof}
  For any  $G$ in $C^{1,2}_0([0,T]\times[-1,1])$,    see \eqref {Jb1},  we have 
\begin{equation} \begin{split}
\cj_G^\b(\rz)& = -\frac\b2 \int_0^T dt \Big \< \s(\rz_t) \nabla (\jn\star \rz_t) \,,\, \nabla G_t \Big \>\\
& -
 \frac 12 \int_0^T dt
 \big\langle \sigma( \rz_t ), \big( \nabla G_t \big)^2 \big\rangle \; \cr & \le 
  \frac{\b^2}8\int_0^T dt\big\<\s(\rz_t),\big[\nabla (\jn\star \rz_t)\big]^2\big\>, 
 \end{split}
\end{equation} 
by inequality \eqref{sch-ineq}, taking $a=1$.    The solution of \eqref{f00} belongs to $L^2([0,T],H^1(\L))$ and its time derivative belongs to $L^2([0,T],H^{-1}(\L))$.
Therefore,
\begin{equation*}
\begin{aligned}
&\hat I_T^\b(\rho^0 |\gamma) \, =\,
\sup_{G\in C_0^{1,2} ([0,T]\times[-1,1])} 
\Big\{ \frac\b2 \int_0^T dt \Big \< \s( \rho^0_t) \nabla (\jn\star  \rho^0_t) \,,\, \nabla G_t \Big \>\\
&\qquad\qquad\qquad\qquad\qquad\qquad\qquad
\;-\; \frac 12 \int_0^T dt
\big \< \s ( \rho^0_t), (\nabla G_t)^2 \big\> \Big\}\;.
\end{aligned}
\end{equation*}
By inequality \eqref{sch-ineq}, this last expression is bounded by 
$$
\frac{\b^2}{ 8 }\int_0^T dt\big\<\s(\rho^0_t),\big[\nabla (\jn\star  \rho^0 _t)\big]^2\big\>\, .
$$
We conclude the proof by applying Lemma \ref{lem-jn} and Jensen inequality.
\end{proof}

\subsection{$I_T^\b(\cdot|\gamma)$-Density.}\label{I-density}

In this section we show that any trajectory $\pi\in D([0,T],\mc M) $, with finite rate
function, $I_T^\b(\pi|\gamma)<\infty$, can be approximated by a
sequence of smooth trajectories $\{\pi^n : n\ge 1\}$ such that
\begin{equation*}
\lim_{n\to\infty}\pi^n\, =\, \pi\ \ \text{in}\ \  D([0,T], \mc M)\quad\text{and}\quad 
\lim_{n\to\infty}I_T^\b(\pi^n|\gamma)  \, =\, I_T^\b(\pi|\gamma)\;.
\end{equation*}

\begin{definition}
A subset $\ca$ of $D([0,T],\mc M)$ is said to be
$I_T^\b(\cdot|\gamma)$-dense if for every $\pi$ in $D([0,T],\mc M)$ such
that $I_T^\b(\pi|\gamma)<\infty$, there exists a sequence $\{\pi^n : n\ge
1\}$ in $\ca$ such that $\pi^n$ converges to $\pi$ in $D([0,T],\cm)$ and
$I_T^\b(\pi^n|\gamma)$ converges to $I_T^\b(\pi|\gamma)$.
\end{definition}

\begin{definition} \label {LT1} Let $\ca_1$ be the subset of $D([0,T],\mc M)$ consisting of trajectories
$\pi$  such that $I_T^\b(\pi|\gamma)<\infty$ and for which there exists $\delta>0$ such that $\pi$ is a weak solution of
the equation \eqref{f00} in the time interval $[0,\delta]$.
\end{definition}

\begin{lemma}
The set $\ca_1$ is $I_T^\b(\cdot|\gamma)$-dense.
\end{lemma}

\begin{proof}
Fix a path $\pi$ such that $I_T^\b(\pi|\gamma)<\infty$ and let $  \rho^{(0)} $ be
the solution of the heat equation \eqref{f00}.  For $\epsilon
>0$, define $\pi^\epsilon$ as
\begin{equation*}
\pi^\epsilon_t (\cdot) \;=\;
\left\{
\begin{array}{ll}
{\displaystyle
  \rho^{(0)}_t(\cdot)} & \text{for $0\le t\le \epsilon$}, \\
{\displaystyle
\rho^{(0)}_{2\epsilon - t}(\cdot) } & \text{for $\epsilon \le t\le 2
  \epsilon$}, \\
{\displaystyle
\pi_{t-2\epsilon}(\cdot) } & \text{for $2 \epsilon \le t\le T$}. 
\end{array}
\right.
\end{equation*}
Since $ \lim_{\ve \to 0}\pi^\ve =\pi$ in $D([0,T],\cm)$ and $I(\cdot|\g)$ is lower semicontinuous, it is enough to
prove that $\forall \varepsilon >0$, $I_T^\b(\pi^\varepsilon|\gamma)<\infty$ and that 
$\liminf_{\varepsilon\to 0} I_T^\b(\pi^\varepsilon|\gamma)\le I_T^\b(\pi|\gamma)$. From Lemma \ref{lem01},
for each $\ve>0$, $\pi^\ve_0(\cdot)=\gamma(\cdot)$ and $\pi^\ve_t(\pm 1)=\rho_\pm$. Decompose the rate function
$I_T^\b(\pi^\ve|\gamma)$ as the sum of the contribution on each interval $[0,\ve]$,  $[\ve,2\ve]$ and $[2\ve,T]$.
Since on $[0,\ve]$ the path $\pi$ satisfies equation \eqref{f00}, by Lemma \ref{lem4.7}, the contribution to 
the first interval is bounded by
$$
\frac{\b^2}8\int_0^\ve dt\int_\L (\nabla \rho^{(0)}_t)^2 (v)dv\, .
$$
This converges to 0 when $\ve\downarrow 0$. 
On the time interval $[\ve,2\ve]$, $\pi^\ve$ satisfies
$$\partial_t\pi^\ve_t = -\partial_t\rho^{(0)}_{2\ve-t}=
-\Delta\rho^{(0)}_{2\ve -t}\,=\, -\Delta\pi^\ve\, .
$$
In particular, the contribution to $[\ve,2\ve]$ is equal to
\begin{equation}\label{l4.9.1}
\begin{aligned}
&\sup
\Big\{2\int_{0}^{\ve}dt\;
\big\< \nabla\rho^{(0)}_t,\nabla G \big\>+ 
\frac\b2 \int_0^\epsilon dt \Big \< \s(\rho^{(0)}_t) \nabla(\jn\star  \rho^{(0)}_t) \,,\, \nabla G_t \Big \>\\
&\qquad\qquad\qquad\qquad\qquad\qquad
\;-\; \frac 12 \int_0^\epsilon dt
\big \< \s (\rho^{(0)}_t), (\nabla G_t)^2 \big\> \Big\}\, ,
\end{aligned}
\end{equation}
where the supremum in taken over all $G\in  C^{1,2}_0([0,T]\times[-1,1])$.
We apply inequality \eqref{sch-ineq} to the first and second term inside the supremum, then apply Lemma \ref{lem4.7}. 
By Lemma \ref{lem-jn}, the supremum \eqref{l4.9.1} is bounded by
\begin{equation*}
\begin{aligned}
&4\int_0^{\ve}dt\int_{\L}du\;
\frac{(\nabla\rho^{(0)}_t)^2}{\sigma(\rho^{(0)}_t)}\,+\,
\frac{\b^2}4\int_0^\ve dt\big\<\s(\rho^{(0)}_t)[\nabla(\jn\star \rho^{(0)}_t)]^2\big\>\\
&\qquad\qquad
\le 4\int_0^{\ve}dt\int_{\L}du\;
\frac{(\nabla\rho^{(0)}_t)^2}{\sigma(\rho^{(0)}_t)}\,+\, \frac{\b^2}8\int_0^\ve dt\int_\L (\nabla \rho^{(0)}_t)^2
\, .
\end{aligned}
\end{equation*}
This last expression converges to zero as
$\ve\downarrow 0$. Finally, the contribution on $[2\ve,T]$ is bounded by
$I_T^\b(\pi|\gamma)$.
\end{proof}

\begin{definition} \label {LT2} 
Denote by $\ca_2$ the subset of $\ca_1$ of all trajectories $\pi$
such that for all $0<\delta \le T$, there exists $\epsilon >0$ such
that $\epsilon \le \pi_t(u) \le 1-\epsilon$ for $(t,u)\in [\delta, T]\times [-1,1]$.
\end {definition}
\begin{lemma}
\label{g04}
The set $\ca_2$ is $I_T^\b(\cdot|\gamma)$-dense.
\end{lemma}

\begin{proof}
By the previous lemma, it is enough to show that each trajectory
$\pi$ in $\ca_1$ can be approximated by trajectories in
$\ca_2$.  Fix $\pi$ in $\ca_1$ and let $ \rho^0 $ be the solution of the equation \eqref{f00}. For each $0<\varepsilon\le1$, let
$\pi^{\varepsilon}=(1-\varepsilon)\pi+\varepsilon \rho^0 $.   We have that $\pi^{\varepsilon} $ converges to $\pi$ as $\varepsilon\downarrow 0$ a.e. in $\L \times (0,T)$.   
By the lower
semicontinuity of $I_T^\b(\cdot|\gamma)$, 
it is enough to show that
 \begin{eqnarray}\label{ls2}
\sup_{\ve>0}{I_T^\b(\pi^\ve|\gamma)}\,<\infty\quad \text{and}\ \ \ 
\limsup_{\ve\to0}I_T^\b(\pi^{\varepsilon}|\gamma)\leq I_T^\b(\pi|\gamma)\, .
\end{eqnarray}
Fix $\ve>0$, by construction $\pi^\ve_0(\cdot)=\gamma$ and $\pi^\ve_t(\pm 1)=\rho_\pm$ for almost all $t\in[0,T]$. 
{}From the convexity of $\mc Q(\cdot)$, for each $0\le\ve\le 1$, 
$$\mc Q(\pi_\epsilon) \le (1-\ve) \mc Q(\pi) +
 \ve \mc Q (\rho^0)\le \mc Q(\pi)+ \mc Q(\rho^0)<\infty\, .
$$
Since 
$\partial_t \pi^\ve=\ve\partial_t \rho^0  +(1-\ve)\partial_t\pi$ and,  by the assumption,   $ I_T^\b(\pi|\gamma) $  is finite, it follows from Lemma \ref{g05}, that 
$\pi^\ve\in L^2(0,T;H^1(\L))$ and $\partial_t \pi^\ve \in  L^2(0,T;H^{-1}(\L))  $. 
   Next we show that  $ I_T^\b(\pi^\ve|\gamma)$ is finite uniformly on $\e$.   
We  decompose the rate $I_T^\b(\pi^\ve|\gamma)$ in  two  terms:
\begin{equation}
\label{den1}
 I_T^\b(\pi^\ve|\gamma) 
\;\le\;   A_1 + A_2 
\end{equation}
where 
\begin{equation}
\label{den1a}
A_1 = \sup
\Big\{\int_{0}^T dt\;
\big\< \partial_t \pi^\ve, G_t \big\>
\,+\,\int_{0}^T dt\;
\big\< \nabla\pi^\ve_t, \nabla G_t \big\>
\;-\; \frac 14 \int_0^T dt
\big \< \s (\pi^\ve_t), (\nabla G_t)^2 \big\> \Big\}\\
\end{equation}
\begin{equation}
\label{den1b} A_2 =  \sup
\Big\{\;
\frac\b2 \int_0^T dt \Big \< \s(\pi^\ve_t) \nabla(\jn\star  \pi^\ve_t) \,,\, \nabla G_t \Big \>
\;-\; \frac 14 \int_0^T dt
\big \< \s (\pi^\ve_t), (\nabla G_t)^2 \big\> \Big\}\; ,
\end{equation}
 and  the supremum is taken over $G\in C_0^{1,2}([0,T]\times [-1,1])$. By concavity of $\sigma(\cdot)$,  the term $A_1$
is bounded above by
\begin{equation}  \label {etl2}
\begin{aligned} 
&(1-\ve) \sup
\Big\{\int_{0}^T dt\;
\big\< \partial_t\pi_t , G_t \big\>
\,+\,\int_{0}^T dt\;
\big\< \nabla\pi_t, \nabla G_t \big\>
\;-\; \frac 14 \int_0^T dt
\big \< \s (\pi_t), (\nabla G_t)^2 \big\> \Big\}\\
&
+\ve \sup
\Big\{\int_{0}^T dt\;
\big\< \partial_t  \rho^0, G_t \big\>
\,+\,\int_{0}^T dt\;
\big\< \nabla \rho^0_t, \nabla G_t \big\>
\;-\; \frac 14 \int_0^T dt
\big \< \s (\rho^0_t) \nabla G_t,  \nabla G_t \big\> \Big\}\, .
\end{aligned}
\end{equation}
Since $\rho^0$ solves the heat equation, the second line of the last expression is equal to zero, while the  first line  is
bounded above by $ 2I_T^0(\pi|\gamma)$ which is bounded by Lemma \ref{compar-I}.
  By Schwartz inequality,
 Lemma \ref{lem-jn} and Jensen inequality  
\begin{equation} \label {etl1}
A_2 \le \frac{\b^2}4\int_0^T dt\big\<\s(\pi^\ve_t)\big(\jn\star  |\nabla \pi^\ve_t|\big)^2\big\>
\, \le\,  \frac{\b^2}4 \int_0^T dt\int_\L \big(|\nabla \pi_t|^2 +|\nabla  \rho^0_t |^2 \big)\; .
\end{equation}
By \eqref  {etl2} and \eqref {etl1} we have that  $\displaystyle \sup_{\ve >0} I_T^\b(\pi^\ve|\gamma) <\infty$.

We are now going to prove $\limsup_{\ve\to 0}I_T^\b(\pi^\ve|\gamma)\,\le\,I_T^\b(\pi|\gamma)$.
  By definition of $\pi^\ve$, we have  that, for any $G\in C_0^{1,2}([0,T]\times [-1,1])$,
\begin{equation} \label {entl3}
\int_0^T dt \big< \partial_t \pi^\ve,G_t \big>  =  (1- \e) \int_0^T dt \left \{ \big< \partial_t \pi,G_t \big> + \e \big< \partial_t   \rho^0,G_t \big> \right \}.\end{equation}
 By Lemma \ref{lem05}, there exists $F\in H^1_0(\sigma(\pi))$ such that
$\pi$ solves the boundary value problem  \eqref{f05}.  Taking this into account and  \eqref   {entl3} we  can write
\begin{equation} \label {entl4} 
\begin {split} &
\int_0^T dt \big< \partial_t \pi^\ve,G_t \big>\\
&=\int_0^T dt  \Big\{ \big< -\nabla \pi^\ve,\nabla G_t \big> +
(1-\ve)\Big<  \sigma(\pi_t)\Big[ \frac{\b}2 \nabla(\jn\star \pi) +\nabla F_t \Big],\nabla G_t\Big>\Big\} \cr & =
\int_0^T dt  \Big\{ \big< -\nabla \pi^\ve,\nabla G_t \big> +  \Big< \FF_t^\b(\pi^\ve), \nabla G_t\Big>   +
\Big< \frac{\b}2 \sigma(\pi^\ve_t)\nabla(\jn\star \pi^\ve_t),   \nabla G_t\Big> 
\Big\}, 
\end {split} 
\end{equation}
where  
\begin{equation*}
\FF_t^\b(\pi^\ve) = (1-\ve)\sigma(\pi_t)\Big[ \frac{\b}2 \nabla(\jn\star \pi_t) +\nabla F \Big]
 -\frac{\b}2 \sigma(\pi^\ve_t)\nabla(\jn\star \pi^\ve_t)\; .
\end{equation*}
By the definition of $\hat I_T^\b $, see \eqref {2:Ib}, we have that 
\begin{equation}    \label {etm1}
\hat I_T^\b(\pi^{\varepsilon}|\gamma) = 
\sup
\Big\{\int_0^T dt\big\<  \FF_t^\b(\pi^\ve)\, ,\,\nabla G_t \, \big\>
  \;-\; 
\frac{1}{2}\int_0^T dt\big\< \sigma(\pi^{\varepsilon}_t)\, ,\, 
(\nabla G_t)^2 \, \big\>\Big\}\,  =  I_T^\b(\pi^{\varepsilon}|\gamma),
\end{equation}
where the supremum  is taken over all $G\in \mc C^{1,2}_0([0,T]\times \L)$.  The last equality  in \eqref  {etm1} holds because  $\mc Q(\pi_\epsilon) $ is bounded for any $\e>0$ and then  \eqref {3:Ib} applies.  
 Since for any $\e>0$ there exists $\delta (\e)>0$ so that  $\sigma(\pi^{\varepsilon}) \ge  \delta (\e)$ we can apply to the first term inside the argument of the supremum of \eqref   {etm1}   inequality \eqref{sch-ineq}  to cancel the  contribution of the second term  inside the argument of the supremum, obtaining 
$$
I_T^\b(\pi^{\varepsilon}|\gamma)\,\le\, 
\frac{1}{2}\int_0^T dt\int_\L du\;\frac{( \FF_t^\b(\pi^\ve)
)^2}
{\sigma(\pi^{\varepsilon}_t)}\;  .
$$
On the other hand, from \eqref{f06}, 
$$
I_T^\b(\pi|\gamma)\,=\, 
 \frac{1}{2}\int_0^T dt\;\big<\sigma(\pi_t)\nabla F_t,\nabla F_t\big> \; .
$$
Therefore, to conclude the proof, it is enough to show that  
\begin{equation} \label {eq:t1}
\lim_{\ve\to 0} \frac{1}{2}\int_0^T dt\int_\L du\;\frac{(\FF_t^\b(\pi^\ve)
)^2} {\sigma(\pi^{\varepsilon}_t)}
\, =\, \frac{1}{2}\int_0^T dt\;\big<\sigma(\pi_t)\nabla F_t,\nabla F_t\big>\; .
\end{equation}
By the continuity of $\sigma$ and the definition of $\FF_t^\b(\pi^\ve)$, 
\begin{equation*}
\lim_{\varepsilon \to 0} \frac{(\FF_t^\b(\pi^\ve)
)^2} {\sigma(\pi^{\varepsilon}_t)}
=\sigma(\pi_t)\big(\nabla F_t\big)^2,  \quad a.e.   \quad  \hbox {in} \quad \L \times (0,T). 
\end{equation*}
 By the convexity of   $a\to a^2$, the inequality \eqref{sch-ineq}, the following inequality
$$
(a+b+c)^3\le 3(a^2 +b^2+c^2),\quad \forall a,b,c\in\R\, ,
$$
the concavity of $\sigma(\cdot)$ and Lemma \ref{lem-jn}, for any $0<\ve<1$,
\begin{equation*}
\begin{aligned}
\frac{(\FF_t^\b(\pi^\ve)
)^2} {\sigma(\pi^{\varepsilon})} \, &
\le \frac{3\b^2}{4}\sigma(\pi) \big[ J^{neum}\star (\nabla\pi)^2 \, + (\nabla F)^2 \big]\\ 
\ & \ \ 
+ \frac{3\b^2}{2}\big[\sigma(\pi)+\sigma(\rho^0)\big] \big[ J^{neum}\star \big( (\nabla\pi)^2 +(\nabla  \rho^0 )^2\big) \big]\; .
\end{aligned}
\end{equation*}
 Therefore \eqref {eq:t1}   follows  by  Lebesgue dominated convergence theorem. \end{proof}

\begin{definition} \label {LT3} 
Denote by $\ca_3$ the 
   trajectories $\pi \in  D([0,T],\mc M)$
such that $\pi$ is the solution of the  boundary value problem  \eqref{f05} for some $F\in  \mc C^{1,2}_0([0,T]\times \L)$.
\end {definition}
The last step is to prove that $\ca_3$ is $I_T^\b(\cdot|\gamma)$-dense.
We follow the strategy adopted in \cite{flm}: given a trajectory  $\pi$ in $D([0,T],\mc M)$  with finite
rate function $I_T^\b(\pi|\gamma)<\infty$,  from Lemma \ref {lem05}, there exists a function
$F$ in $H^1_0(\sigma(\pi))$ such that $\pi$ is a weak solution to the equation \eqref{f05}. Instead of approximating
$\pi$ by a sequence of smooth trajectories in $D([0,T],\mc M)$ (cf. \cite{blm},\cite{mo1},\cite{qrv}), we will approximate $F$ by smooth
functions $(F_n)$ and we then show that the corresponding smooth solutions $(\pi_n)$ of \eqref{f05} converge to $\pi$ 
in  $D([0,T],\mc M)$ and $I_T^\b(\pi^n|\gamma)$ converges to $I_T^\b(\pi|\gamma)$.
 
 \begin{lemma}
\label{lg05b}
The set $\ca_3$ is $I_T^\b(\cdot|\gamma)$-dense.
\end{lemma}

\begin{proof}
In view of the previous lemma, it is enough
to show that for each $\pi$ in $\ca_2$, we can exhibit a sequence $\{\pi_n : n >0\}$ in $\ca_3$
which converges to $\pi$ in $D([0,T],\mc M)$ and such that $I_T^\b(\pi_n|\gamma)$
converges to $I_T^\b(\pi|\gamma)$. Fix $\pi\in\ca_2$. Since $I_T^\b(\pi| \gamma)$  is finite,  by Lemma \ref{lem05}, there exists a function
$F\in H^1_0(\sigma(\pi))$ such that $\pi$ is  the  weak solution to the  boundary value problem  \eqref{f05}. We claim hat
$F\in L^2\big( [0,T],H^1(\L)\big)$ and then, can be approximated by a sequence of smooth functions $(F^n)_{n\ge 1}$. Indeed,
let $0<\d<T $  be such that, 
$\pi$ is the solution of the heat equation \eqref{f00} in the time interval $[0,\d]$. We have that 
$\nabla F =-\frac{\b}2 \nabla (\jn\star \pi)$ in $[0,\d]\times \L$ and
\begin{equation}\label{g05a}
\begin{aligned}
\int_0^T dt\int_\L \big|\nabla F_t(u)\big|^2 du & \, =\, 
\int_0^\d dt\int_\L \frac{\b^2}4 \big|\nabla (\jn\star \pi)(t,u)\big|^2 du\\
\ &\ + \int_\delta^T dt\int_\L \big|\nabla F_t(u)\big|^2 du\, .
\end{aligned}
\end{equation}
On the other hand, since $\pi\in\ca_2$, there exists $0 <\ve<1$ such
that $\ve\le \pi_t(\cdot) \le 1-\ve$ for $\d \le t\le T$. Therefore
\begin{equation}\label{g05b}
\begin{aligned}
\int_\delta^T dt\int_\L \big|\nabla F_t(u)\big|^2 du &\le
\frac{1}{\sigma(\ve)}\Vert F\Vert_{\sigma(\pi)}^2\\
\ &\ =
\frac2{\sigma(\ve)} I_T^\b(\pi|\gamma) <\infty \, .
\end{aligned}
\end{equation}
It follows from \eqref{g05a} and \eqref{g05b} that $F\in L^2\big( [0,T],H^1(\L)\big)$.
Let $(F^n)_{n>0}$ be a sequence of functions in $C^{1,2}_0([0,T]\times \L)$ such that $\displaystyle\lim_{n\to +\infty} F^n=F$ in
$L^2\big( [0,T],H^1(\L)\big)$.

For each integer $n>0$, let $\pi^n$ be the weak solution of
\eqref{f05} with $F^n$ in place of $F$.    By \eqref{f06}
\begin{equation*}
I_T^\b(\pi^n|\gamma) = \frac{1}{2}\int_0^T\ dt\;
\langle \nabla F^n_t \cdot \sigma(\pi^n_t) \nabla F^n_t\rangle
\leq \frac{1}{4} \int_0^T dt\int_{\L} du\;
\Vert\nabla F^n_t(u)\Vert^2\, .
\end{equation*}
Since the sequence $(F_n)_{n>0}$ converges to $F$ in $L^2\big( [0,T],H^1(\L)\big)$, it follows from the last inequality that
$I_T^\b(\pi^n|\gamma)$ is uniformly bounded. Thus, by Theorem \ref{th4}, the sequence $\pi^n$ is relatively compact in $D([0,T],\mc M)$.
  Let $\{\pi^{n_k}:\, k\geq 1\}$ be a subsequence of $\pi^n$ converging
to some $\pi^0$ in $D([0,T],\mc M)$, then $\{\pi^{n_k}:\, k\geq 1\}$ converges weakly to $\pi^0$ in $L^2\big( [0,T]\times [-1,1]\big)$.   
Since $I_T^\b(\pi^n|\gamma)$ is uniformly bounded, by Lemma \ref{g05} and Lemma \ref{g06}, $\pi^{n_k}$ converges to
$\pi^0$ strongly in $L^2([0,T]\times [-1,1])$.
For every $G$ in $\mc
C^{1,2}_0([0,T]\times [-1,1])$, we have
\begin{equation*}
\begin{split}
& \langle\pi^{n_k}_T,G_T\rangle - \langle\gamma,G_0\rangle = 
\int_0^T dt\;\langle\pi^{n_k}_t,\partial_tG_t\rangle\\
&\qquad\qquad\qquad
+\int_0^T dt\;\langle \pi^{n_k}_t),
\Delta G_t\rangle
\; -{\rho_+}   \int_0^{T} \! dt\, \nabla G_t(1) 
\; +\;  {\rho_-}  \int_0^{T} \!dt\, \nabla G_t(-1)\\
&\qquad\qquad\qquad
\,  +\int_0^T \langle \nabla G_t,\sigma(\pi_t^{n_k})
\big[ \frac{\b}{2}\nabla (\jn\star \pi_t^{n_k}) + \nabla F_t^{n_k}\big]\rangle \, dt.
\end{split}
\end{equation*}
Letting  $k\to \infty$, we obtain
\begin{equation*}
\begin{split}
& \langle\pi^{0}_T,G_T\rangle - \langle\gamma,G_0\rangle = 
\int_0^T dt\;\langle\pi^{0}_t,\partial_tG_t\rangle\\
&\qquad\qquad\qquad
+\int_0^T dt\;\langle \pi^{0}_t,
\Delta G_t\rangle
\; -{\rho_+}   \int_0^{T} \! dt\, \nabla G_t(1) 
\; +\;  {\rho_-}  \int_0^{T} \!dt\, \nabla G_t(-1)\\
&\qquad\qquad\qquad
\,  +\int_0^T \langle \nabla G_t,\sigma(\pi_t^{0})
\big[ \frac{\b}{2}\nabla (\jn\star \pi_t^{0}) + \nabla F_t\big]\rangle \, dt.
\end{split}
\end{equation*}
That is $\pi^0$ is a  weak solution of equation \eqref{f05}. Thus, by uniqueness of weak solutions of \eqref{f05},
$\pi^0=\pi$.

To conclude the proof of the lemma it remains to prove that $\lim_{n\to \infty}I_T^\b(\pi^n|\gamma)=I_T^\b(\pi|\gamma)$.
The sequence $(\pi^n)_{n>0}$ converges to $\pi$ strongly in $L^2([0,T]\times [-1,1])$ and the sequence $(F^n)_{n>0}$ converges to
$F$ in $L^2([0,T],H^1(\Lambda))$.  Taking into account that $\pi$ is bounded and $\sigma$ is Lipschitz, we obtain
\begin{equation*}
\begin{aligned}
\lim_{n\to \infty} I_T^\b(\pi^n|\gamma)&
=\lim_{n\to \infty} \frac{1}{2}\int_0^T\ dt\;
\langle \nabla F^n_t \cdot \sigma(\pi^n_t) \nabla F^n_t\rangle\\
\ & \  =\frac{1}{2}\int_0^T\ dt\;
\langle \nabla F_t \cdot \sigma(\pi_t) \nabla F_t\rangle\; =I_T^\b(\pi|\gamma)\, .
\end{aligned}
\end{equation*}

\end{proof}
\subsection{Upper bound}
Let $\tq=\tq^2$ be the functional  defined in  \eqref{ccc} with $\delta_0=2$.
For  $0\le a\le 1$, denote by $\mc E_a:D([0,T],\mc M)\to [0,+\infty]$ the following functional
$$
\mc E_a(\pi)\; =\; \hat I_T^\b(\pi|\g)  \; +\; a(1+a) \tq(\pi)\; .
$$
The proof of the upper bound relies on the following proposition.
\begin{proposition}
\label{up1}
Let  ${\mc K}$ be a compact set of $D([0,T],\mc M)$. There exists a positive constants $C$, such that for any $0\le a\le 1$,
\begin{equation*}
\limsup_{N\to\infty} \frac 1{N^d}  \log
 Q^{\b,N}_{\eta^N} ({\mc K})\;\le\; -\frac1{1+a}\inf_{\pi\in {\mc K}}\mc E_a(\pi)\, +\, a C (T+1)\; .
\end{equation*}
\end{proposition}

\begin{proof}
Fix a density profile $\t: [-1,1]\to (0,1)$, a function $G$ in
$C^{1,2}_0([0,T]\times [-1,1])$ and a function $H$ in
$C^\infty_c([0,T]\times \L)$. 
For a local function $\Psi:\{0,1\}^{\bb Z}\to \R$,   $c>0$ and $\epsilon >0$,   denote $B_{N,\epsilon,c}^{G,\Psi}$ and $E_{N,c}^{G}$  the set of trajectories $(\eta_t)_{t\in [0,T]}$ defined by
\begin{eqnarray*}
B_{N,\epsilon,c}^{G,\Psi} &=& \Big\{ \eta_\cdot \in D([0,T], {\mc S}_N) : 
\Big| \int_0^T V_{N,\epsilon}^{G,\Psi}(t,\eta_t) dt \Big| \le c\Big\} \;,\\
E_{N,c}^{G}&=& \Big\{ \eta_\cdot \in D([0,T], {\mc S}_N) : 
\Big| \int_0^T W_N^{G}(t,\eta_t) dt \Big| \le c\Big\} \;,
\end{eqnarray*}
where $V_{N,\epsilon}^{G,\Psi}$ and $W_N^G$ are defined in \eqref{vne} and \eqref{vneb}.

Define $\Psi_1 (\eta) = [\eta(1)- \eta(0)]^2$, let $A_{N,\epsilon,c}^{G,H}$ be the set
\begin{equation*}
A_{N,\epsilon,c}^{G,H} \;=\;
B_{N,\epsilon,c}^{\nabla G,\Psi_1}\cap
E_{N,c}^{\nabla G}\cap B_{N,\epsilon,c}^{\nabla H,\Psi_1}\; . 
\end{equation*}
By the superexponential estimates 
stated in Proposition \ref{see1} and Proposition \ref{seeb}, it is enough to prove that, for every $0<a\le1 $, 
\begin{equation*}
\limsup_{c\to 0}
\limsup_{\epsilon\to 0}\limsup_{N\to\infty} \frac 1{N}  \log
 Q^{\b,N}_{\eta^N} \Big({\mc K} \cap A_{N,\epsilon,c}^{G,H} \Big)\;\le\; 
-\frac1{1+a}\inf_{\pi\in {\mc K}}\mc E_a(\pi)\, +\,  a C  (T+1)\; .
\end{equation*}
For $H\in \mc C_c^\infty([0,T]\times \L)$, recall from \eqref{cccg} the definition of 
${\tilde{\mc Q}}_H(\pi)={\tilde{\mc Q}}_H^2(\pi)$, with $\d_0=2$ and write
\begin{eqnarray*}
&&\frac 1{N}  \log
 Q^{\b,N}_{\eta^N} \Big({\mc K} \cap A_{N,\epsilon,c}^{G,H} \Big)
\; =\;\\
&&\qquad
\frac 1{N}  \log
  \Es_{\eta^N}^{\b,N}\Big[ \1\{{\mc K}\cap A_{N,\epsilon,c}^{G,H} \}e^{aN{\tilde{\mc Q}}_H(\pi^N * \iota_\epsilon) } 
e^{-aN{\tilde{\mc Q}}_H(\pi^N * \iota_\epsilon)}\Big]\;.
\end{eqnarray*}
By H\"older inequality the right hand side of the last equality is bounded above by
\begin{eqnarray}\label{aab}
&& \frac{1}{(1+a)N} \log
  \Es_{\eta^N}^{\b,N}   \Big[ \1\{{\mc K}\cap A_{N,\epsilon,c}^{G,H} \}e^{-a(1+a)N{\tilde{\mc Q}}_H(\pi^N* \iota_\epsilon)}\Big]\\
&&\quad
+ \; \frac{a}{(1+a)N} \log
 \Es_{\eta^N}^{\b,N} \Big[ e^{(1+a)N{\tilde{\mc Q}}_H(\pi^N* \iota_\epsilon)}\Big]\; .\nonumber
\end{eqnarray}
{}From Lemma \ref{hs02}, the second term of this inequality is bounded by $aC_1 (T+1)$.
Consider the exponential martingale $M^G_t$ defined by
\begin{equation}   \label{Mtg1}
 \begin {split} &
M^G_t \;=\; \exp \Big\{ N \Big[ \<\pi^N_t, G_t\> 
- \<\pi^N_0, G_0\> \cr
 &  
\qquad\qquad\qquad\qquad\qquad
\; -\; \frac 1{N} \int_0^t e^{- N \<\pi^N_s, G_s\>} (\partial_s +N^{2} \mc L_N) 
\, e^{N \<\pi^N_s, G_s\>} \,ds \Big] \Big\}\; .
\end {split} \end{equation}
Since the sequence $\{\eta^N : N\ge 1\}$ is associated to $\gamma$, an
elementary computation shows that on the set $A_{N,\epsilon,c}^{G,H}$
\begin{equation}\label{Mtg}
M^G_T \;=\;  \exp N\Big\{ {\mc J}_G^\b(\pi^N * \iota_\epsilon) 
\; +\; O_G (\epsilon)\;+\; O(c) \Big\} \; 
\; ,
\end{equation}
where $O_{G} (\epsilon)$ (resp. $O(c)$) is a quantity which
vanishes as $\epsilon \downarrow 0$ (resp. $c \downarrow 0$) and $ {\mc J}_G^\b(\cdot)$ is the functional defined in \eqref {Jb1}.  Consider the first term of \eqref{aab}
and rewrite it as
$$
\frac{1}{(1+a)N} \log
 \Es_{\eta^N}^{\b,N} \Big[ M^G_T(M^G_T)^{-1}\1\{{\mc K}\cap A_{N,\epsilon,c}^{G,H} \}e^{-a(1+a)N{\tilde{\mc Q}}_H (\pi^N *  \iota_\epsilon)}\Big]
$$
Optimizing over $\pi^N$ in ${\mc K}$, since $M^G_t$ is a mean one positive
martingale, the previous expression is bounded above by
\begin{equation*}
-  \frac1{1+a}\inf _{\pi\in {\mc K}} \Big\{{\mc J}_G^\b(\pi * \iota_\epsilon) 
\; +\ a(1+a)\tilde{\mc Q}_H(\pi *  \iota_\epsilon)\Big\}
\;+\; O_G (\epsilon) \;+\; O(c) \; .
\end{equation*}
Optimize the
previous expression with respect to $G$ and $H$.
Since the set ${\mc K}$ is compact and  ${\mc J}_G^\b(\cdot *  \iota_\epsilon)$ and $\tilde{\mc Q}_H(\cdot *  \iota_\epsilon)$
are lower semi-continuous for every $G$, $H$, $\epsilon$, we
may apply the arguments presented in \cite[Lemma 11.3]{v} to exchange
the supremum with the infimum.  In this way we obtain that the last
expression is bounded above by
$$
\frac1{1+a}\sup _{\pi\in {\mc K}}\inf_{G,H,\epsilon} \Big\{-{\mc J}_G^\b(\pi*\iota_\epsilon) 
\; -\ a(1+a)\tilde{\mc Q}_H(\pi * \iota_\epsilon)\Big\}
\;+\; O_G (\epsilon) \;+\; O(c) \; .
$$
Letting first
$\epsilon\downarrow 0$ and then $c \downarrow 0$, we obtain that the limit of 
the previous expression is bounded above by
$$
 \frac1{1+a}\sup _{\pi\in {\mc K}} \inf_{G,H}\Big\{ -{\mc J}_G^\b(\pi ) 
\; -\ a(1+a)\tilde{\mc Q}_H(\pi)\Big\}
$$
This concludes the proof of the proposition because
$\sup_{G} {\mc J}_{G}^\b (\pi) = \hat I_T^\b(\pi|\g)$.
\end{proof}

\noindent{\it Proof of the upper bound.} Let ${\mc K}$ be a compact set of $D([0,T],\mc M)$. If for all $\pi\in {\mc K}$, 
$\tilde{\mc Q}(\pi)=\infty$ then the upper bound is trivially satisfied. 
Suppose that $\inf_{\pi \in {\mc K}}\big\{\tilde{\mc Q}(\pi)\big\}<\infty$, from Proposition \ref{up1}, for any $0< a\le 1$,
\begin{eqnarray*}
&&\limsup_{N\to\infty} \frac 1{N}  \log
 Q_{\eta^N}^{\b,N} \big({\mc K}\big)\;\le\;
 -\frac1{1+a}\inf_{\pi\in {\mc K},  \tilde{\mc Q}(\pi)<\infty}\mc E_a(\pi)\, +\,  a C  (T+1) \\
&&\qquad\qquad\qquad
=\; -\frac1{1+a}\inf_{\pi\in {\mc K}}\Big\{ I_T^\b (\pi|\gamma)\, +\,a(1+a)\tilde{\mc Q}(\pi) \Big\}
\, +\, a C  (T+1)\\
&&\qquad\qquad\qquad
\le\; -\frac1{1+a}\inf_{\pi\in {\mc K}}\big\{I_T^\b (\pi|\gamma)\big\}\, -\,
a\inf_{\pi\in {\mc K}}\tilde{\mc Q}(\pi)
\, +\,  a C  (T+1) \; .
\end{eqnarray*}
To conclude the proof of the upper bound for compact sets, it remains to let $a\downarrow 0$.

To pass from compact sets to closed sets, we have to obtain
exponential tightness for the sequence  $Q_{\eta^N}^{\b,N}$.
 The proof presented in \cite[Section 10.4.]{kl} is easily adapted to our context.

\cqfd

\subsection{Lower bound}
In this section we   establish the large deviations lower bound. 

The strategy of the proof of the lower bound  consists  of two
steps. We first prove that for each $\pi\in \ca_3$, recall  its definition  in \ref {LT3}, and each neighborhood $\mc N_\pi$ of
$\pi$ in  $D \big( [0,T], \mc M \big)$
\begin{equation}\label{low1}
\liminf_{N \to \infty}\;\frac1N \log Q_{\eta^N}^{\b,N} \big\{ \mc N_\pi \big\} \ge - I_T^\b(\pi|\g)\, .
\end{equation}
The proof of the lower bound is then accomplished
by showing, see Subsection \ref{I-density}, that for any $\pi\in D \big( [0,T], \mc M\big)$
with $I_T^\b(\pi|\g) <\infty$ we can find a sequence of $\pi^k \in \ca_3$ 
such that $\lim_{k \to \infty } \pi^k =\pi$ in $D \big( [0,T], \mc M \big)$ and
$\lim_{k \to \infty }  I_T^\b\big(\pi^k|\g\big) = I_T^\b(\pi|\g)$. 

The proof of   \eqref{low1} is similar to the one in the periodic case, see \cite[Section 10.5.]{kl}. It
depends on establishing laws of large numbers, in hydrodynamic scaling, for weak
perturbations of the original process, and controlling by the Girsanov formula the relative entropies of the processes that go
with these perturbations. Fix a path $\pi\in \ca_3$.   Then, by construction,  there exists 
     $F\in \mc C_0^{1,2}\big( [0,T] \times [-1,1]\big)$ so  that $\pi$ is the weak solution of the equation \eqref{f05}.    
Recall from \eqref{Mtg1} the definition of the exponential martingale $M^F_T$. 
Let $\bb P_{\eta^N}^F$ be the probability measure on the path space $D([0,T], \mc S_N)$ with density $M^F_T$ with respect to  $\bb P^{\beta,N}_{\eta^N}$:
$\bb P_{\eta^N}^F[A]=  \bb E^{\beta,N}_{\eta^N}\big[M^F_T\1\{A\} \big]$.
Let $(\eta_t)_{t\in [0,T]}$  be the process with law $\bb P_{\eta^N}^F$ on  $D([0,T], \mc S_N)$.  Let $(\pi_t^N)_{t\in [0,T]}$  be the corresponding 
empirical measure. Then $(\pi_t^N)_{t\in [0,T]}$ converges weakly in probability to $(\pi_t)_{t\in [0,T]}$. From the super 
exponential estimates, Proposition \ref{see1} and Proposition \ref{seeb}, we have
$$
\liminf_{N\to\infty}\frac1N \log Q_{\eta^N}^{\b,N} \big\{ \mc N_\pi \big\} \ge -\lim_{N\to\infty}\frac{1}{N}H
\left(\bb P^F_{\eta^N}\big|\bb P^{\beta,N}_{\eta^N}\right)\, ,
$$
where $H\big(\bb P^F_{\eta^N}\big|\bb P_{\eta^N}^{\b,N}\big)$ stands for the relative entropy given by
\begin{equation*}
H\big( \bb P^F_{\eta^N}\big|\bb P^{\beta,N}_{\eta^N} \big) \; =\; \int  
\log  \Big\{ \frac{d\bb P^F_{\eta^N} }{d\bb P^{\beta,N}_{\eta^N} } \Big\} \,d\bb P^F_{\eta^N} \; .
\end{equation*}
To conclude the proof, it remains to show that
\begin{equation*}
\lim_{N\to\infty}\frac{1}{N}H
\left(\bb P^F_{\eta^N}\big|\bb P^{\beta,N}_{\eta^N}\right) = I_T^\b(\pi|\gamma)\, .
\end{equation*}
The proof of \cite[Theorem 10.5.4]{kl} is easily adapted to our model.
\qed

 \section{Appendix}
 
 In this    section  we   summarize  the  properties  of  the equation   \eqref {eq:1} needed to prove the main results of the paper.  
 The  proofs of these results are based on applying  standard tools in partial differential equations, 
 although some care need to be taken because of  the presence of the nonlocal term.
     Notice that   because of  the  nonlocal term   the  comparison property does not hold 
 for this equation, so   tools based on maximum principle   will  not work  for  \eqref {eq:1}. 
 
 We recall   the notion of  weak solution of \eqref  {eq:1}. 
A   function $\rho(\cdot, \cdot ):[0, T]\times \L\to [0,1] $    is a weak solution   
of the initial-boundary value problem  \eqref  {eq:1}  if    $ \r \in  L^2\big([0,T],H^1(\Lambda)\big)$ and   for every $ G \in    C^{1,2}_0([0,T]\times[-1,1])$  one has    $ \ell_G^\b(\rho, \gamma)=0,$
where $\ell_G^\b$ was defined in \eqref  {lb1}.

 \medskip
\begin{theorem} \label{eru}  For any $\beta\ge0$ there exists an unique weak solution of  \eqref  {eq:1}.
 \end {theorem}
 \medskip  The existence  of a weak  solution of \eqref  {eq:1} is   a consequence of the the tightness of   $(Q_{\mu_N}^{\b})_{N\geq 1}$ and the characterization of the support of  its limit points, see Lemma \ref {lem3}.    The uniqueness can be  easily proven  performing estimates as in Theorem \ref {stat2} for all $\beta$. 
A  proof of existence  without invoking the hydrodynamic limit  can be done applying in our setting the argument  done in   \cite {gl2}, Section 4.    
 
   \medskip
\begin{theorem} \label{stat2}  There exists  $ \beta_0$  depending on $ \jn$ and $\Lambda$, so that for $\beta \le \b_0$   there exists an unique weak  stationary solution $\bar \rho$  of  \eqref {eq:1s}. 
 Further,  let $ \rho_t (\rho_0)  $ be the  weak solution of  \eqref {eq:1}
with     initial  datum $\rho_0 \in  \mc M $. For $ \beta <\beta_0$,  there exists $c(\beta)>0$ so that  
$$   \|  \rho_t (\rho_0)- \bar \rho   \|_{L^2 (\Lambda)} \le e^{-c(\beta) t}  \| \rho_0 - \bar \rho   \|_{L^2 (\Lambda)}.  $$ 
 \end {theorem}
 \begin {proof} 
Let   $\rho_{i,0} \in   \mc M $   and     $ \rho_{i,t}$    be the      solution   of  \eqref {eq:1} for  $t \ge 0$ , with initial datum $\rho_{i,0}$,       $i=1,2$.  Set $v =  \rho_1-\rho_2$, 
we have
\begin {equation} \label {J1}   
\begin {aligned}   
&\frac 12 \frac d {dt} \|v_t\|^2_{L^2}  =   
 - \int_\L   \nabla v \left [ \nabla  v-     \beta \chi (\rho_{1})  \nabla (\jn \star \rho_{1})+  \beta \chi (\rho_{2})
 \nabla (\jn\star \rho_{2})   \right ]  \\
&\qquad\qquad
 = - \int_\L   \nabla v \left [      \nabla  v -     \beta \chi (\rho_1)  \nabla (\jn\star v) +  
\beta [   \chi (\rho_2) -   \chi (\rho_1)]  \nabla  (\jn \star   \rho_2)   \right ]. 
\end {aligned}   
\end {equation} 
Since $ \chi (a) \le \frac 14$ for $a \in [0,1]$  and  $ \left | \chi (\rho_2) -   \chi (\rho_1) \right | =  \left |  [\rho_2- \rho_1] \left [ 1- (\rho_2+ \rho_1) \right]  \right |  \le   |v| $   we have that
\begin {equation} \label {J2}    
   \int_\L   \nabla v \left [      \nabla v-     \beta \chi (\rho_1)  \nabla(\jn \star v)  \right ] \ge  \|  \nabla v \|_{L^2}^2 [ 1 - \frac \beta 4] 
     \end {equation} 
\begin {equation} \label {J3}     \begin {split}   &
  \left |\int_\L   \nabla v \left [     \chi (\rho_2) -   \chi (\rho_1)]   \nabla   (\jn \star\rho_2)
  \right ] \right |   \le   \sup | \nabla  \jn|   \|  \nabla  v \|_{L^2}  \| v \|_{L^2} \cr & \le 
  \frac a 2   \sup | \nabla  \jn|^2  \|  \nabla  v \|_{L^2}^2 + \frac 1 {2a} \| v \|_{L^2}^2,
   \end {split}  \end {equation} 
   for any $a>0$. 
   Taking into  account \eqref  {J2}  and \eqref  {J3} we can estimate \eqref  {J1} as following:
   \begin {equation} \label {J4}    
\frac 12 \frac d {dt} \|v\|^2_{L^2} \le   -    \|  \nabla  v \|_{L^2}^2 [ 1 -  \beta ( \frac 1 4 +    
  \frac a 2   \sup | \nabla  \jn|^2 )] +   \frac \b {2a}  \|  v \|_{L^2}^2.
  \end {equation} 
   and we choose $a$  so that
  $$  \frac 1 4 +    
  \frac a 2   \sup | \nabla  \jn|^2   \le \frac 1 3 $$
  Since we are in a bounded domain we can use the  Poincar\'e  inequality
   $$  \| v \|_{L^2}^2 \le C (\Lambda) \|  \nabla  v \|_{L^2}^2 $$
  obtaining
    \begin {equation} \label {J4b}    
\frac 12 \frac d {dt} \|v\|^2_{L^2} \le   -     \| v \|_{L^2}^2 \frac { [ 1 -  \frac \b 3 ]} {C(\Lambda)}    +  \frac \b {2a} \|  v \|_{L^2}^2.
  \end {equation} 
  Take $ \beta_0$ so   that
   $$ \frac { [ 1 -  \frac {\b_0} 3 ]} {C(\Lambda)} -  \frac {\b_0} {2a} =0 $$
   Then for $ \beta <\beta_0$ there exists $ c(\b) >0$, $ \frac { [ 1 -  \frac \b 3 ]} {C(\Lambda)} -  \frac \b {2a} =  c(\b)>0 $, so that 
   \begin {equation} \label {J5}  
\frac 12 \frac d {dt} \|v\|^2_{L^2}   
  \le  - c(\b) \|  v \|_{L^2}^2.
 \end {equation}    
This implies immediately that the stationary solution is unique and that it is exponential attractive in $L^2$.  
 \end{proof}

 \bigskip
\noindent {\bf Acknowledgements}.  The authors wish to thank Professor Daniel Cahill 
(Department of Communication, University of Rouen) for reading and correcting the English in the first version of the paper.
Part of this work was done during the  Enza Orlandi stay at 
 Fields Institute (Toronto, Canada)   (whose hospitality is acknowledged), for the  
Thematic Program on Dynamics and Transport in Disordered Systems: January-June 2011.
 Mustapha Mourragui  thanks the warm hospitality of the Universit\`a di Roma Tre.

\end{document}